\documentclass{amsart}
\usepackage{stmaryrd}
\usepackage[all]{xy}
\usepackage{mathrsfs}
\usepackage{amscd,amssymb,color}

\newcommand{\Sdot}[1][\ssdot]{S_{#1}}

\newcommand{\iSdot}[1][\ssdot]{S^{\infty}_{#1}}

\mathchardef\varDelta="7101

\newcommand{\C}{\mathbf{C}}

\let\sma\wedge
\newcommand{\htp}{\simeq}
\renewcommand{\to}{\mathchoice{\longrightarrow}{\rightarrow}{\rightarrow}{\rightarrow}}

\newcommand{\cA}{{\mathcal A}}
\newcommand{\cB}{{\mathcal B}}
\newcommand{\cC}{{\mathcal C}}
\newcommand{\cD}{{\mathcal D}}

\newcommand{\cM}{{\mathcal M}}
\newcommand{\cS}{{\mathcal S}}

\newcommand{\cU}{{\mathcal U}}

\let\catsymbfont\mathcal
\newcommand{\aA}{{\catsymbfont{A}}}

\newcommand{\aC}{{\catsymbfont{C}}}
\newcommand{\aD}{{\catsymbfont{D}}}
\newcommand{\aE}{{\catsymbfont{E}}}

\newcommand{\aS}{{\catsymbfont{S}}}
\newcommand{\aT}{{\catsymbfont{T}}}

\renewcommand{\C}{\catsymbfont{C}}

\renewcommand{\L}{\mathrm{L}}

\newcommand{\bbZ}{\mathbb{Z}}
\newcommand{\bbN}{\mathbb{N}}

\newcommand{\bN}{{\mathbb{N}}}

\newcommand{\bZ}{{\mathbb{Z}}}

\def\quickop#1{\expandafter\DeclareMathOperator\csname
#1\endcsname{#1}}
\quickop{id}\quickop{Id}\quickop{colim}\quickop{hocolim}\quickop{op}
\quickop{co}\quickop{Ar}\quickop{sing}\quickop{Hom}\quickop{w}\quickop{Ho}
\quickop{ob}\quickop{diag}\quickop{Stab}\quickop{Cat}\quickop{Mot}\quickop{Mod}
\quickop{map}\quickop{Cone}\quickop{End}\quickop{Idem}\quickop{Perf}\quickop{Ind}
\quickop{Gap}\quickop{Shv}\quickop{Spaces}\quickop{cof}\quickop{ex}\quickop{perf}
\quickop{tri}\quickop{Set}\quickop{Fib}\quickop{Nat}\quickop{haut}\quickop{holim}\quickop{Tor}\quickop{Ext}\quickop{Diff}\quickop{lax}\quickop{Mult}\quickop{Ob}\quickop{cpt}\quickop{CAlg}\quickop{cf}\quickop{f}\quickop{KEnd}\quickop{Ch}\quickop{hofib}\quickop{KAut}

\newcommand{\lkan}[1]{(#1)_{\mathrm{iso}}}
\newcommand{\N}{\mathrm{N}}

\newcommand{\icat}{\Cat_\i}

\newcommand{\istabcat}{\Cat_\i^{\ex}}

\newcommand{\idemstabcat}{\Cat_\i^{\perf}}
\newcommand{\idemtimes}{\otimes^{\vee}}
\newcommand{\idemfun}{\Fun^{\ex}}

\newcommand{\Fun}{\mathrm{Fun}} 
\newcommand{\Map}{\mathrm{Map}}
\newcommand{\bbS}{\mathbb{S}}

\newcommand{\ispec}{\aS}

\newcommand{\Spt}{\mathcal{S}}

\newcommand{\Alg}{\mathrm{Alg}}

\newcommand{\sk}{\mathrm{sk}}

\newcommand{\ie}{{i.e.,}\ }

\newcommand{\too}{\longrightarrow}

\renewcommand{\i}{\infty}
\newcommand{\rep}{\mathrm{rep}} 

\newcommand{\Motadd}{\cM_{\mathrm{add}}}

\newcommand{\Umot}{\cU_{\mathrm{add}}}

\DeclareMathOperator{\Proj}{Proj}

\numberwithin{equation}{section}

\newtheorem{theorem}[equation]{Theorem}
\newtheorem*{theorem*}{Theorem}

\newtheorem{corollary}[equation]{Corollary}
\newtheorem{lemma}[equation]{Lemma}
\newtheorem{proposition}[equation]{Proposition}
\theoremstyle{definition}
\newtheorem{definition}[equation]{Definition}

\newtheorem{remark}[equation]{Remark}
\newtheorem{example}[equation]{Example}
\newtheorem{notation}[equation]{Notation}

\xyoption{arrow}
\xyoption{matrix}
\xyoption{cmtip}
\SelectTips{cm}{}

\newdir{ >}{{}*!/-5pt/\dir{>}}

\begin{document}

\title[$K$-theory of endomorphisms via noncommutative motives]{$K$-theory of endomorphisms\\ via noncommutative motives}

\author{Andrew J. Blumberg}
\address{Department of Mathematics, University of Texas,
Austin, TX \ 78703}
\email{blumberg@math.utexas.edu}
\thanks{A.~J.~Blumberg was partially supported by the NSF grant DMS-0906105}
\author{David Gepner}
\address{Fakult\"at f\"ur Mathematik,
Universit\"at Regensburg, 93040 Regensburg, Germany}
\email{djgepner@gmail.com}
\thanks{}
\author{Gon{\c c}alo~Tabuada}
\address{Department of Mathematics, MIT, Cambridge, MA 02139}
\email{tabuada@math.mit.edu}
\thanks{G.~Tabuada was partially supported by the NEC award-2742738}


\begin{abstract}
In this article we study the $K$-theory of endomorphisms using
noncommutative motives. We start by extending the $K$-theory of
endomorphisms functor from ordinary rings to (stable)
$\infty$-categories. We then prove that this extended functor $\KEnd(-)$
not only descends to the category of noncommutative motives but
moreover becomes co-represented by the noncommutative motive
associated to the tensor algebra $\bbS[t]$ of the sphere spectrum
$\bbS$.  Using this co-representability result, we then classify all
the natural transformations of $\KEnd(-)$ in terms of an integer plus a
fraction between polynomials with constant term $1$; this solves a
problem raised by Almkvist in the seventies.  Finally, making use of
the multiplicative co-algebra structure of $\bbS[t]$, we explain how
the (rational) Witt vectors can also be recovered from the symmetric
monoidal category of noncommutative motives.  Along the way we show
that the $K_0$-theory of endomorphisms of a {\em connective} ring
spectrum $R$ equals the $K_0$-theory of endomorphisms of the
underlying ordinary ring $\pi_0R$.
\end{abstract}

\maketitle
\vskip-\baselineskip
\vskip-\baselineskip
\vskip-\baselineskip
\setcounter{tocdepth}{1}
\tableofcontents
\vskip-\baselineskip
\vskip-\baselineskip
\vskip-\baselineskip

\section{Introduction}
\subsection*{$K$-theory of endomorphisms}
The $K$-theory of endomorphisms was introduced in the seventies by
Almkvist \cite{Almkvist,Almkvist1} and
Grayson \cite{Grayson,Graysonwitt}.  Given an ordinary ring $A$, let
${\bf P}(A)$ be the category of finitely generated projective (right)
$A$-modules and $\End({\bf P}(A))$ the associated category of
endomorphisms: its objects are the pairs $(M,\alpha)$, with $M \in
{\bf P}(A)$ and $\alpha$ an endomorphism of $M$, and its morphisms
$(M,\alpha) \to (M',\alpha')$ are the $A$-linear homomorphisms
$f:M \to M'$ verifying the equality $f \alpha=\alpha'f$.  Note that
this latter category inherits naturally from ${\bf P}(A)$ an exact
structure in the sense of Quillen \cite{Quillen}.  The classical {\em
$K$-theory of endomorphisms of $A$} was then defined as the homotopy
groups of the (connective) algebraic $K$-theory spectrum $\KEnd({\bf
P}(A))$ of the exact category $\End({\bf P}(A))$.

Bloch~\cite{Bloch}, and later Stienstra~\cite{Stienstra, Stienstra1},
related the $K$-theory of endomorphisms to crystalline
cohomology.  More recently, work of
Betley-Schlichtkrull \cite{BetleySchlichtkrull},
Hesselholt \cite{Hesselholtptypical}, and
Lindenstrauss-McCarthy \cite{LindenstraussMcCarthy}, establishes
precise connections between the $K$-theory of (parametrized)
endomorphisms, Goodwillie calculus, and trace methods in  
algebraic $K$-theory.  In this article, we study the more foundational
aspects of the $K$-theory of endomorphisms using noncommutative motives. 
\subsection*{Noncommutative motives}
Let $\idemstabcat$ be the $\infty$-category of small idempotent-complete
stable $\infty$-categories; see \S\ref{s:infty}.  Standard examples
are the $\infty$-category $\Perf_R$ of perfect modules over a ring
spectrum $R$ and the $\infty$-category $\Perf_X$ of perfect
complexes for a scheme $X$.  Recall 
from~\cite[6.1]{BGT} that a functor 
$E\colon \idemstabcat \to \cD$, with values in a stable presentable
$\infty$-category $\cD$, is called an {\em additive invariant} if it
preserves filtered colimits and sends split-exact sequences of
$\infty$-categories to (necessarily split) cofiber sequences of
spectra. Examples include algebraic $K$-theory (see \S\ref{s:kinf}),
topological Hochschild homology ($THH$), and topological cyclic
homology ($TC$).  In~\cite[\S 6]{BGT} we have constructed the {\em
universal additive invariant}  
\begin{equation}\label{eq:universal}
\Umot \colon \idemstabcat \too \Motadd\,.
\end{equation}
Given any stable presentable $\infty$-category $\cD$, there is an induced equivalence
\begin{equation}\label{eq:induced}
(\Umot)^\ast \colon \Fun^\L(\Motadd,\cD) \stackrel{\sim}{\too} \Fun_\mathrm{add}(\idemstabcat,\cD)
\end{equation}
where the left-hand side denotes the $\infty$-category of
colimit-preserving functors and the right-hand side the
$\infty$-category of additive invariants. Because of this property,
which is reminiscent of motives, $\Motadd$ is called the category of
{\em noncommutative motives}.  As with any stable $\infty$-category,
$\Motadd$ carries a natural enrichment $\Map(-,-)$ in spectra;
see~\cite[\S 4.2]{BGT}. In \cite[\S 7.3]{BGT} we proved that for every 
$\cC\in \idemstabcat$ there is a natural equivalence of spectra 
\begin{equation}\label{eq:corep1}
\Map(\Umot(\Perf_{\bbS}), \Umot(\cC)) \simeq K(\cC)\,,
\end{equation}
where $\bbS$ stands for the sphere spectrum.  Intuitively speaking,
algebraic $K$-theory becomes co-represented by the noncommutative
motive associated to $\bbS$.
\subsection*{Statements of results}
Given an $\i$-category $\aC$, we start by defining the {\em
$\i$-category $\mathrm{End}(\aC)$ of endomorphisms in $\cC$} as the
functor $\i$-category
$\End(\aC):=\Fun(\Delta^1/\partial\Delta^1,\aC)$; see
definition \ref{def:K-th-endo}. By first restricting this construction
to $\idemstabcat$ and then by applying the algebraic $K$-theory
functor, we obtain a well-defined $K$-theory of endomorphisms functor 
\begin{equation}\label{eq:key}
\KEnd \colon \idemstabcat \to \cS_\infty
\end{equation}
with values in the $\infty$-category of symmetric spectra.  Our first
main result characterizes this functor as follows:

\begin{theorem}{(see theorem \ref{thm:co-rep})}\label{thm:main1}
The above functor \eqref{eq:key} is an additive
invariant.  Moreover, for every $\cC \in \idemstabcat$ there is a
natural equivalence of spectra
\begin{equation}\label{eq:corep2}
\Map(\Umot(\Perf_{\bbS[t]}), \Umot(\cC)) \simeq \KEnd(\cC)\,,
\end{equation}
where $\bbS[t]$ stands for the tensor algebra of $\bbS$. 
\end{theorem}

Intuitively speaking, theorem~\ref{thm:main1} shows us that the
functor \eqref{eq:key} not only descends to $\Motadd$ but moreover
becomes co-represented by the noncommutative motive associated to
$\bbS[t]$.  By adding a ``formal variable'' $t$ to $\bbS$ one passes
then from algebraic $K$-theory \eqref{eq:corep1} to $K$-theory of
endomorphisms \eqref{eq:corep2}.  The following result justifies the
correctness of our construction.
 
\begin{theorem}{(see theorem~\ref{thm:agreement2})}\label{thm:main2intro}
Given an ordinary ring $A$, the associated spectrum $\KEnd(\Perf_{HA})$
(where $HA$ denotes the Eilenberg-MacLane ring spectrum of $A$)
is naturally equivalent to $\KEnd({\bf P}(A))$.
\end{theorem}

As explained by Almkvist in \cite[page~339]{Almkvist}, a very
interesting problem in the $K$-theory of endomorphisms is the
classification of all the natural transformations of the functor
$A \mapsto \KEnd({\bf P}(A))$.  Classical examples include the
Frobenius $F_n$ and the Verschiebung $V_n$ operations.  This problem
was studied in the particular case of the $K_0$-theory of
endomorphisms by Hazewinkel~\cite{Hazewinkel1} and many operations in
the higher $K$-theory of endomorphisms were computed by
Stienstra \cite{Stienstra,Stienstra1}.  In \S\ref{sec:operations} we
extend $F_n$ and $V_n$ to the $\infty$-categorical setting and (making
use of theorem~\ref{thm:main1}) solve the problem raised by Almkvist
as follows: given a (commutative) ring $A$, let us write $W_0(A)$ for
the multiplicative (abelian) group
\begin{equation}\label{eq:form}
\left\{\frac{1+a_1r + \cdots + a_ir^i  +\cdots +
a_nr^n}{1+b_1r+ \cdots + b_jr^j + \cdots + b_mr^m} \,|\, a_i,b_j \in
A \,\, n,m \geq 0 \right\}\,.
\end{equation}
\begin{theorem}{(see theorem~\ref{thm:natural})}\label{thm:main3}
There is a canonical weak equivalence of spectra
\begin{equation}\label{eq:eq-spectra}
\Nat(\KEnd,\KEnd) \htp \KEnd(\Perf_{\bbS[t]}),
\end{equation}
where $\Nat$ stands for the spectrum of natural transformations.
Moreover, the group $\pi_0 \Nat(\KEnd,\KEnd)$ of natural
transformations up to homotopy is isomorphic to
\begin{equation}\label{eq:ab-groups}
\pi_0 \KEnd(\Perf_{\bbS[t]}) \simeq \bbZ \oplus W_0(\bbZ[t])\,.
\end{equation}
Furthermore, under the above identifications, the
Frobenius operations $F_n$ correspond to the elements $(1,1+r^nt)$ and
the Verschiebung operations $V_n$ to the elements $(n,1+rt^n)$.
\end{theorem}

Roughly speaking, Theorem~\ref{thm:main3} shows us that all the
information concerning a natural transformation of the above
functor \eqref{eq:key} can be completely encoded in an integer plus a 
fraction between polynomials with constant term $1$. It shows also us
that the Frobenius (resp. the Verschiebung) operation is the
``simplest one'' with respect to the variable $t$ (resp. $r$). In
order to prove the above isomorphism \eqref{eq:ab-groups} we
have made the following computation which is of general interest:

\begin{theorem}{(see theorem \ref{thm:endcomp})}\label{thm:main}
For every {\em connective} ring spectrum $R$, one has an isomorphism
$\pi_0 \KEnd(\Perf_{R}) \simeq K_0(\End({\bf P}(\pi_0 R)))$ of abelian
groups.
\end{theorem} 

One interesting feature of the $K$-theory of endomorphisms is its
connection with Witt vectors.  Given a {\em commutative} ring $A$, the
{\em Witt ring $W(A)$} of $A$ is the abelian group of all power series
of the form $1+ a_1r+ a_2 r^2 + \cdots$, with $a_i \in A$, endowed
with the multiplication $\ast$ determined by the equality
$(1-a_1r)\ast (1-a_2r) = (1-a_1a_2r)$. The {\em rational} Witt ring
$W_0(A) \subset W(A)$ of $A$ consists of the elements of the
form \eqref{eq:form}. As observed by Grayson~\cite{Grayson}, $W_0(A)$
is a dense $\lambda$-subring of $W(A)$ and hence $W(A)$ can be
recovered from $W_0(A)$ by a completion procedure.

As proved in~\cite[\S4]{BGT1}, the standard symmetric monoidal
structure $-\otimes^\vee-$ on $\idemstabcat$ can be extended to
$\Motadd$ in a universal way making \eqref{eq:universal} symmetric
monoidal.  This additional structure allows us to recover the rational
Witt ring (and hence the Witt ring) from the category of
noncommutative motives as follows:

\begin{theorem}{(see \S\ref{sec:Witt})}\label{thm:new-intro}
The ring maps $\bbS[t]\stackrel{t \mapsto
t \wedge t}{\to} \bbS[t] \wedge \bbS[t]$ and
$\bbS[t] \stackrel{t=1}{\to} \bbS$ induce a counital
coassociative cocommutative coalgebra structure on
$\Perf_{\bbS[t]} \in \idemstabcat$. Moreover, the ring maps
$\bbS \to \bbS[t]$ and
$\bbS[t] \stackrel{t=0}{\to} \bbS$ give rise to a wedge sum
decomposition
$\Umot(\Perf_{\bbS[t]}) \simeq \Umot(\Perf_{\bbS}) \vee \mathbb{W}_0$ of
counital coassociative cocommutative coalgebras in $\Motadd$.
\end{theorem}

Using theorem~\ref{thm:new-intro}, we then obtain a well-defined lax
symmetric monoidal functor $\Map(\mathbb{W}_0,-)$ from $\Motadd$ to
$\cS_\infty$ which we call the {\em rational Witt ring spectrum
functor}.  This terminology is justified by the following agreement
result: 

\begin{theorem}{(see theorem~\ref{thm:decomposition})}\label{thm:main4}
For every ordinary commutative ring $A$ one has a ring isomorphism 
\begin{equation}\label{eq:Witt}
\pi_0 \Map(\mathbb{W}_0, \Umot(\Perf_{HA}))\simeq W_0(A)\,.
\end{equation}
\end{theorem}

Isomorphism \eqref{eq:Witt} provides a conceptual characterization of
the rational Witt ring, as the left-hand side is defined solely in
terms of universal properties.  Roughly speaking, by only keeping
track of the ``formal variable'' $t$ of $\bbS[t]$ one passes from
$K$-theory of endomorphisms \eqref{eq:corep2} to rational Witt
vectors \eqref{eq:Witt}.


Finally, making use of Lurie's resolution of Mandell's conjecture
(see \cite[8.1.2.6]{HA}), we obtain the following result: 

\begin{corollary}{(see corollary~\ref{cor:last})}\label{cor:last-intro}
Let $R$ be an $E_n$ ring spectrum. Then the associated rational Witt
ring spectrum $\Map(\mathbb{W}_0,\Umot(\Perf_{R}))$ is an $E_{n-1}$
ring spectrum. 
\end{corollary}

Intuitively speaking, corollary~\ref{cor:last-intro} shows us that the 
rational Witt ring spectrum functor $\Map(\mathbb{W}_0,-)$ decreases
commutativity by one.

\subsection*{Acknowledgments}
The authors are very grateful to Clark Barwick, Lars Hesselholt and
Mike Mandell for useful discussions. 

\section{Preliminaries}
\subsection{Notations}
Given an ordinary ring $A$, we will denote by $\Ch(A)$ the category of cochain
complexes of (right) $A$-modules. We will assume that $\Ch(A)$ is
endowed with the projective Quillen model structure;
see \cite[\S2.3]{Hovey}. The associated homotopy category, \ie the
derived category of $A$, will be denoted by $\cD(A)$. We will write
$\mathrm{perf}(A)$ for the category of {\em perfect} complexes of
$A$-modules, \ie the full subcategory of $\Ch(A)$ consisting of those
complexes that become compact in the derived category
$\cD(A)$.  We will write $\Ch^b(-)$ for the full subcategory of
bounded complexes.  Finally, we write $\Perf_{A}$ to denote the
$\i$-category of perfect complexes over $A$ (and more generally the
category of compact modules over a stable $\i$-category).

\subsection{The homotopy theory of a Waldhausen category}\label{ss:howald}

An important source of categories with weak equivalences
is provided by Waldhausen categories. Recall from \cite{Wald}
that this consists of a category $\aC$ endowed with a 
subcategory of weak equivalences $w\aC$ and with a subcategory of
cofibrations $\cof(\aC)$ such that the pushouts along cofibrations
exist, the cobase change of a cofibration is a cofibration, and
pushout along cofibrations are homotopy pushouts.
Following~\cite{BM}, we will impose further hypotheses that afford
control on the underlying homotopy category of a Waldhausen category.
These come in two forms:
\begin{enumerate}
\item We want to ensure that $\aC$ has
a homotopy calculus of left factions (HCLF) in the sense of
Dwyer-Kan~\cite{DwyerKan}.  Categories with weak equivalences that
have homotopy calculi of fractions admit concise and tractable models
for the mapping spaces in the Dwyer-Kan simplicial localization
$L\aC$~\cite{DwyerKan, DwyerKan2}.  Specifically, we can use concise
models of $L^H \aC$ to represent the homotopy types of the mapping
spaces as the nerves of certain categories of
zig-zags~\cite{DwyerKan2}.

For this purpose, we
impose factorization hypotheses.  Recall from \cite[2.6]{BM} the
notion of a {\em Waldhausen category with functorial mapping
cylinders}.  Such cylinders allow a functorial factorization $A \to
Tf \to B$ of every map $f \colon A \to B$. The map $A \to Tf$ is a
cofibration and $Tf \to B$ comes equipped with a natural section
$B \to Tf$ (which is a weak equivalence).  We will say that a
Waldhausen category admits functorial factorization if every map
$f \colon A \to B$ admits a functorial factorization as a cofibration
followed by a weak equivalence; note that this implies functorial mapping
cylinders (by factoring the fold map $A \coprod
A \stackrel{\nabla}{\to} A$). It is 
possible to weaken our factorization hypotheses to remove the
hypotheses of functoriality; see for instance~\cite[\S A]{BM}. In
fact, it is enough to require functorial mapping cylinders for {\em
weak cofibrations}, \ie maps that are equivalent via a zig-zag of weak
equivalences to a cofibration; see~\cite[\S 2.1]{BM}.  Given
functorial mapping cylinders for weak cofibrations, the category $\aC$
has a homotopy calculus of left fractions; see \cite[5.5]{BM}.

\item We want to ensure the weak equivalences are compatible with the homotopy category of $\aC$.  Recall from \cite{DHKS} that a
category $\aC$ with weak equivalences is called {\em DHKS-saturated}
if a map is a weak equivalence if and only if it is an isomorphism in
the homotopy category~\cite{DHKS}.  In a Waldhausen category with
functorial mapping cylinders for weak cofibrations, the property of
being DHKS-saturated is equivalent to the weak equivalences satisfying
the two out of six property~\cite[6.4]{BM}.  In fact, it suffices that
the weak equivalences are closed under retracts (which is easier to
check).
\end{enumerate}
\subsection{$\infty$-categories}\label{s:infty}
Throughout the article we will assume that the reader is familiar with
the basics of the theory of $\infty$-categories.   For technical
convenience, we work in the setting of Joyal's quasicategories, but
nothing about our work depends on any particular model of
$\infty$-categories.  Standard references for quasicategories material
are \cite{HTT,HA}.  We briefly review a few aspects of the theory of
particular note for our treatment.

There are a number of options for producing the ``underlying''
$\i$-category of a category equipped with a notion of ``weak
equivalence''. The most structured setting is that of a simplicial
model category $\aC$, where the $\i$-category can be obtained by
restricting to the full simplicial subcategory $\aC^{\cf}$ of
cofibrant-fibrant objects and then applying the simplicial nerve
functor $\N$.  More generally, if $\aC$ is a category equipped with a
subcategory of weak equivalences $w\aC$, the Dwyer-Kan simplicial
localization $L \aC$ provides a corresponding
simplicial category, and then $\N((L\aC)^{\textrm{fib}})$, where
$(-)^{\textrm{fib}}$ denotes fibrant replacement in simplicial
categories, yields an associated $\i$-category.  
Lurie has given a version of this approach in \cite[\S 1.3.3]{HA}: we
associate to a (not necessarily simplicial) category $\aC$ with weak
equivalences $W$ an $\i$-category $\N(\aC)[W^{-1}]$; when $\aC$ is a
model category, for functoriality reasons it is usually convenient to
restrict to the cofibrant objects $\aC^{\mathrm{c}}$ and consider
$\N(\aC^{\mathrm{c}})[W^{-1}]$.

Recall from \cite[\S2.2]{BGT} the
notions of {\em stable} $\infty$-category and {\em
idempotent-complete} stable $\infty$-category, and
from \cite[\S2.3]{BGT} the notion of {\em compact} $\infty$-category.
Let us denote by $\Cat_\infty$ the $\infty$-category of small
$\infty$-categories, by $\istabcat$ the $\infty$-category of small
stable $\infty$-categories, and by $\idemstabcat$ the
$\infty$-category of small idempotent-complete stable
$\infty$-categories. As explained in \cite[\S2.2-2.3]{BGT}, the inclusions of subcategories $ \idemstabcat\subset\istabcat\subset\Cat_\infty$ admit left adjoints
\[
\Stab\colon \Cat_\infty\too\istabcat,\qquad\Idem\colon \istabcat\too\idemstabcat.
\]
Note that the inclusion $\istabcat\subset\idemstabcat$ is fully faithful, but $\istabcat\subset\Cat_\infty$ is not.

\subsection{Spectral categories}\label{sub:spectral}
Recall from \cite[\S2]{BM2}\cite[Appendix~A]{SS} that a small {\em
spectral category} is a category enriched over the symmetric monoidal
category $\cS$ of symmetric spectra. As explained in \cite[\S3]{BGT} there
is a close connection between spectral categories and
$\infty$-categories. The category of small spectral categories
$\Cat_\cS$ admits a Quillen model structure with weak equivalencse the
Morita equivalences \cite{Spectral}, and the $\infty$-category
$\idemstabcat$ is equivalent to the $\infty$-category associated to
this model category; see \cite[3.20-3.20]{BGT}. Hence, in order to
simplify the exposition, we will make no notational distinction
between a spectral category (for instance a ring spectrum) and the
associated $\infty$-category. 

\subsection{Algebraic $K$-theory of exact and Waldhausen
categories}

We assume the reader has some familiarity with Quillen's $K$-theory of
exact categories~\cite[\S2]{Quillen} and with Waldhausen's $K$-theory
of categories with cofibrations and weak equivalences (Waldhausen
categories)~\cite[\S 1]{Wald}.  For each of these classes of input
data, we can associate a connective algebraic $K$-theory spectrum.
Given an exact category, we can regard
it as a Waldhausen category with weak equivalences the isomorphisms
and cofibrations the admissible monomorphisms; see~\cite[\S 1.9]{Wald}
for the agreement between the two possible constructions of $K$-theory
in this situation.

\subsection{Algebraic $K$-theory of $\i$-categories}\label{s:kinf}

In this section we quickly review the analogue of the definition of
Waldhausen $K$-theory in the setting of $\i$-categories.  

\begin{definition}\label{defn:infcof}
Let $\aC$ be a pointed $\i$-category.  We say that $\aC$ is an
$\i$-category with cofibrations if we have the additional data of a
subcategory $\cof(\aC)$ such that 
\begin{enumerate}
\item for any object $x$ in $\aC$, the unique map $* \to x$ is a
cofibration,
\item the subcategory $\cof(\aC)$ contains all the equivalences,
\item pushouts along cofibrations exist, and the cobase change of a
cofibration is again a cofibration. 
\end{enumerate}
We refer to the maps in $\cof(\aC)$ as cofibrations.
\end{definition}

Such an $\i$-category with cofibrations is referred to as a Waldhausen
$\i$-category in~\cite{Barwick,Fiore}.  Specifying a subcategory of
cofibrations is a way of specifying which maps have homotopy cofibers.
In many natural examples, all maps are cofibrations.  For instance,
any pointed category with pushouts admits the structure of an
$\i$-category with cofibrations where all maps are cofibrations.

There is a close connection between cofibrations in Waldhausen
$\i$-categories and the weak cofibrations studied in ~\cite{BM} (and
reviewed in \S~\ref{ss:howald} above).  More precisely, we have
the following consistency check: 

\begin{lemma}
Let $\aC$ be a Waldhausen category with weak equivalences $w\aC$ and
cofibrations $\cof(\aC)$ that satisfy the hypothesis of factorization
of weak cofibrations. Let us denote by $\cof(\N\aC[W^{-1}])\subseteq\N\aC[W^{-1}]$ the subcategory of those arrows which are equivalent to the image of an
arrow in $\N\cof(\aC)\subseteq\N\aC$. Under these assumptions and notations, $(\N\aC[W^{-1}],\cof(\N\aC[W^{-1}]))$ is an $\i$-category with
cofibrations.
\end{lemma}

\begin{proof}
By \cite[6.2]{BM}, the factorization hypothesis implies that $\aC$ has a homotopy
calculus of left fractions.  Then, using the universal property of the
pushout in an $\i$-category, we can conclude that pushouts along
cofibrations compute pushouts in the underlying $\i$-category. 
\end{proof}

\begin{remark}
In particular, a Waldhausen category with factorization has an
underlying $\i$-category with cofibrations where all maps are
cofibrations (and all pushouts exist).
\end{remark}

Generalizing~\cite[1.2.2.2]{HA}, we have the following version of the
$\Sdot$ construction:

\begin{definition}\label{defn:gap}
Let $(\aC, \cof(\aC))$ be an $\i$-category with cofibrations.  Denote by
$\Gap([n],\aC, \cof(\aC))$ the full subcategory of  
$\Fun(\N(\Ar[n]),\aC)$ spanned by the functors 
$\N(\Ar[n])\to\aC$ such that, for each $i\in I$, $F(i,i)$ is a zero
object of $\aC$, $F(i,j) \to F(i,k)$ is a cofibration for $i \leq
j \leq k$, and for each $i<j<k$, the following square is cocartesian 
\[
\xymatrix{
F(i,j)\ar[r]\ar[d] & F(i,k)\ar[d]\\
F(j,j)\ar[r]     & F(j,k)\,.}
\]
\end{definition}
Following \cite[1.2.2.5]{HA}, we define a simplicial $\i$-category 
$\iSdot \aC$ by the rule $\iSdot[n] \C=\Gap([n],\aC, \cof(\aC))$.
Applying passage to the largest Kan complex levelwise, we obtain a
simplicial space $\lkan{\iSdot \aC}$.  Then $\Omega
|\lkan{\iSdot \aC}|$ is the $\i$-categorical version of Waldhausen's
$K$-theory space.  Furthermore, for each $n$,
$\Gap([n],\aC, \cof(\aC))$ is itself equipped with a suitable
subcategory of cofibrations: we can can iterate
this procedure.  Since $\Gap([0],\aC, \cof(\aC))$ is contractible (with preferred
basepoint given by the point in $\aC$) and $\Gap([1],\aC, \cof(\aC))$ is
equivalent to $\aC$, there is a natural map 
\[
S^1 \sma \lkan{\aC} \to |\lkan{\iSdot \aC}|
\]
given by the inclusion into the $1$-skeleton.  Therefore, the spaces
$|\lkan{(\iSdot)^{n}(\aC)}|$ assemble to form a spectrum $K(\aC)$,
which is evidently the $\i$-categorical analogue of Waldhausen's $K$-theory
spectrum.  Furthermore, it is natural in functors of $\i$-categories
with cofibrations, i.e., functors which preserve zero objects,
cofibrations, and pushouts along cofibrations.

The following results connect this definition to the usual
definition.  This is a generalization of the comparison
of~\cite[7.12]{BGT}, which handles the case when all maps are
cofibrations.

\begin{theorem}\label{thm:comp}
Let $\aC$ be a Waldhausen category with weak equivalences $w\aC$ and
cofibrations $\cof(\aC)$.  Then there is a natural zig-zag of
equivalences connecting $K(\aC)$ and $K(\N \aC[W^{-1}])$.
\end{theorem}

\begin{proof}
This follows from the rectification technique of~\cite[7.11]{BGT}
coupled with the work of~\cite[7.7]{BGT}.
\end{proof}

\subsection{Grothendieck group of $\infty$-categories}

In this last subsection we give an explicit description of the
Grothendieck group ($K_0$) of an $\infty$-category $\aC$ with
cofibrations.  Observe that we can describe $\iSdot[2] \aC$ as the
full subcategory of the $\i$-category of $2$-simplices
$\sigma\colon\Delta^2\to\aC$ with the property that the composite
\[
\Delta^{\{0,2\}}\too\Delta^2\too\aC
\]
is equivalent to zero and the map specified by $\Delta^{\{0,1\}}$ is a
cofibration.
Moreover, definition~\ref{defn:infcof} implies that the subcategory
$\iSdot[2](\aC)$ always contains the ``split exact sequence'' $A\to
A\oplus B\to B$ corresponding to each pair of objects $A$ and $B$ of $\aC$. 

\begin{lemma}
The abelian group $K_0(\aC)$ can be described as the cokernel
\begin{equation}\label{eqn:K0}
\bigoplus_{\pi_0 \iSdot[2](\aC)^\simeq}\bZ\too\bigoplus_{\pi_0\aC^\simeq}\bZ\too
K_0(\aC) 
\end{equation}
of the map which, on the component corresponding to the equivalence
class of the exact sequence $[A\to B\to C]\in\pi_0 S_2(\aC)$, sends
$1\in\bZ$ to the element $[A\oplus C]-[B]\in\bigoplus_{\pi_0\aC}\bZ$.
Here $\pi_0\aC^\simeq$ denotes the set of equivalence classes of
objects of the $\i$-category $\aC$, which is to say the set of
connected components of the underlying $\i$-groupoid $\aC^\simeq$ of
$\aC$.
\end{lemma}

\begin{proof}
The argument is similar to the classical analysis of $K_0$ in the
Waldhausen context.
\end{proof}

\begin{example}
If $\aC$ is a stable $\i$-category, viewed as an $\i$-category with cofibrations by taking all arrows to be
cofibrations, we see that $\iSdot[2](\aC)$ is the full subcategory of
$\Fun(\Delta^2,\aC)$ consisting of the (co)fiber sequences; that is,
those $\sigma\colon \Delta^2\to\aC$ which extend to a cocartesian
square of the form 
\[
\xymatrix{
A \ar[r]\ar[d]\ar[rd]^\sigma & B\ar[d]\\
0\ar[r] & C}
\]
where $0$ denotes a zero object of $\aC$.
Equivalently, these are the $2$-simplices which give rise to
distinguished triangles $A\to B\to C$ in the underlying triangulated
homotopy category. 
\end{example}
\section{$K$-theory of endomorphisms}\label{sec:end}
In this section we extend the $K$-theory of endomorphisms from
ordinary rings to stable $\infty$-categories. The main results are
theorems~\ref{thm:co-rep}-\ref{thm:agreement2} and
proposition~\ref{prop:agreement1}. 

\begin{definition}
Let us denote by $\cD\simeq\underline{\mathbb{N}}$ the category with
one object and endomorphism monoid $\mathbb{N}$.  The category of
endomorphisms in a category $\cC$ is the functor category
$\Fun(\cD,\cC)$; that is, it has 
as objects the pairs $(x,\alpha)$ where $x \in \cC$ and
$\alpha \colon x \to x$ an endomorphism, and morphisms
$(x,\alpha) \to (x', \alpha')$ the maps $f \colon x \to x'$ such
that $f\alpha=\alpha'f$. Note that the unique functor
$*\to\underline{\mathbb{N}}$ induces a forgetful functor
$\End(\cC) \to \cC$.
\end{definition}

\subsection{Exact categories}\label{sub:AG}
Given an exact category $\cC$ in the sense of
Quillen~\cite[\S2]{Quillen}, the category $\End(\cC)$ inherits an
exact structure by declaring a sequence to be exact if its image under
$\End(\cC) \to \cC$ is exact.  The endomorphism $K$-theory of $\cC$ is
then defined as the connective algebraic $K$-theory spectrum
$\KEnd(\cC)$ of the exact category $\End(\cC)$. Clearly this
construction is functorial in exact functors.

Given an ordinary ring $A$, let $\mathbf{P}(A)$ denote the exact
category of finitely-generated projective (right) $A$-modules.
Following Almkvist~\cite{Almkvist,Almkvist1} and
Grayson~\cite{Grayson}, the $K$-theory of endomorphisms of $A$ is
defined as the $K$-groups associated to the spectrum
$\KEnd(\mathbf{P}(A))$.
\subsection{Waldhausen categories}

We now extend the above construction to the setting of Waldhausen
categories.

\begin{lemma}\label{lem:Wald}
The category
$\mathrm{End}(\cC)$ of endomorphisms in a Waldhausen category $\cC$ carries a canonical Waldhausen structure. A morphism is a cofibration (resp. a
weak equivalence) in $\End(\cC)$ if its image under
$\End(\cC) \to \cC$ is a cofibration (resp. a weak equivalence) in
$\cC$. Furthermore, if $\aC$ admits functorial factorization, then so does
$\End(\aC)$.
\end{lemma}

\begin{proof}
Pushouts along cofibrations in $\End(\cC)$ are computed in 
$\cC$, using the induced endomorphism on the pushout.  Hence,
pushouts along cofibrations exist in $\End(\cC)$ and the gluing axiom
holds.  The remaining properties are clear.
\end{proof}

Using Lemma~\ref{lem:Wald}, one can associate (as in the case of
exact categories) to every Waldhausen category $\cC$ a well-defined
algebraic $K$-theory spectrum $\KEnd(\cC)$.  This construction is
clearly functorial in exact functors of Waldhausen categories. 

\begin{example}\label{ex:complexes}
A motivating example is the case in which $\aC=\perf(A)$ is the category of perfect complexes over an ordinary ring $A$. The weak 
equivalences are the quasi-isomorphisms and the cofibrations are the
morphisms which admit retractions as morphisms of graded (right)
$A$-modules (i.e., the degree-wise split monomorphisms).  We have an
equivalence of $\infty$-categories 
\[
\mathrm{N}(\perf(A))[W^{-1}]\simeq\Perf_{HA},
\]
where $HA$ is the associated ring spectrum.  The category
$\End(\perf(A))$ admits an algebraic description.  Specifically,
$\End(\perf(A))$ is the category of those complexes of $A[t]$-modules
that are perfect as underlying complexes of $A$-modules.  Note that
$\End(\perf(A)) \neq \perf(A[t])$ since perfect complexes of
$A[t]$-modules tend not to be perfect as complexes of $A$-modules.
\end{example}

\begin{example}\label{ex:freeassoc}
Example~\ref{ex:complexes} can be generalized to the case where $\cC$
is the category $\Perf_{R}$ of perfect complexes over a ring spectrum
$R$. When $R$ is the Eilenberg-MacLane spectrum $HA$ of an ordinary
ring $A$, we recover example~\ref{ex:complexes}.  Note that in this
case $\End(\Perf_{R})$ also admits an algebraic
description. Specifically, an $R$-module $M$ endowed with an
endomorphism is precisely the same data as a module over $R[\bN]:=
R \sma\Sigma^\infty_+ \bN$. This follows from the fact that
$R \sma\Sigma^\infty_+\bN$ is the free $R$-algebra on one generator,
so that a map to $\End(M)$ is (by adjunction) just a map of
$R$-modules $M \to M$ (see~\cite[II.4.4]{EKMM}
or~\cite[3.10]{Schwede}).
\end{example}
\subsection{Endomorphisms of $\infty$-categories and
co-representability}\label{sec:corep}

The category $\underline{\bbN}$ with one object and endomorphism
monoid $\bbN$ (under addition) is freely generated by a single
nonidentity arrow.  As a consequence, we obtain a well-defined map
$\Delta^1/\partial\Delta^1\to\N(\underline{\bbN})$ which is a
categorical equivalence of simplicial sets.

\begin{definition}\label{def:K-th-endo}
Let $\aC \in \Cat_\infty$ be an $\infty$-category.
\begin{itemize}
\item[(1)] The $\i$-category $\mathrm{End}(\aC)$ of endomorphisms in $\cC$ is
the functor $\i$-category $\End(\aC):=\Fun(\Delta^1/\partial\Delta^1,\aC)$.
Note that, as colimits in functor $\i$-categories are computed
pointwise, $\End(\aC)$ has finite colimits if and only if $\aC$ has
finite colimits. Moreover, if $\cC$ is stable then $\mathrm{End}(\cC)$
is also stable~\cite[1.1.3.1]{HA}.
\item[(2)] If $\aC$ has finite colimits and a zero object, then the
{\em $K$-theory of endomorphisms of $\aC$} is defined as the
(connective) spectrum $\KEnd(\aC)$; see \S\ref{s:kinf}.
\end{itemize}
\end{definition}

Clearly, definition \ref{def:K-th-endo}(ii) is functorial in $\cC$.
Hence, we obtain a well-defined $K$-theory of endomorphisms functor 
\begin{equation}\label{functor:End}
\KEnd\colon \idemstabcat \too \Spt_\infty
\end{equation}
with values in the $\infty$-category of symmetric spectra.  In
this section we prove that \eqref{functor:End} not only descends to 
the category $\Motadd$ but moreover that it becomes co-representable;
see Theorem~\ref{thm:co-rep} below.

\begin{notation}
We denote by $\bbS[t]$ the tensor algebra on the sphere spectrum $\bbS$;
see \cite[Example~3.10]{Schwede}.
\end{notation}

\begin{lemma}\label{cor:compact}
The $\infty$-category $\Perf_{\bbS[t]} \in \idemstabcat$ is compact.
\end{lemma}

\begin{proof}
Recall that there is an equivalence
$\idemstabcat\simeq\mathrm{Pr^L_{St,\omega}}$ induced by passage to
the $\Ind$-category~\cite[\S 5.5.7]{HTT}, where
$\mathrm{Pr}^L_{St,\omega}$ is the 
$\i$-category of compactly-generated stable $\i$-categories.  Thus, it
suffices to show that the $\i$-category of $\bbS[t]$-modules is
compact as a compactly-generated stable $\i$-category, i.e. as an
object of $\mathrm{Pr^L_{St,\omega}}$.  By proposition 3.5
of \cite{AG}, this reduces to showing that $\bbS[t]$ is compact as an
$\bbS$-algebra, which is clear because it is free on one generator.
\end{proof}

\begin{proposition}\label{prop:endident}
Let $\cA$ be an object in $\idemstabcat$.  Then there is a natural
equivalence 
\[
\idemfun(\Perf_{\bbS[t]},\cA)\simeq\End(\cA).
\]
\end{proposition}

\begin{proof}
By the work of~\cite[\S 4]{BGT}, we know that
$\idemfun(\Perf_{\bbS[t]}, \cA)$ can be described as the $\i$-category
associated to the spectral category $\rep(\bbS[t], \tilde{\aA})$ of
$S[t]$-$\tilde{\aA}$-bimodules, where $\tilde{\cA}$ is a spectral lift
of $\aA$.  An $\bbS[t]$-$\tilde{\cA}$-bimodule is the same thing as a
$\tilde{\cA}$-module with an endomorphism.  Next, the condition of
being right compact means that these are precisely the
$\tilde{\cA}$-modules with endomorphisms that are compact as
$\tilde{\cA}$-modules; \ie this is the category
$\End(\perf(\tilde{\cA}))$.
\end{proof}

\begin{theorem}\label{thm:co-rep}
The above functor \eqref{functor:End} is an additive invariant.
Moreover, for every $\cC \in \idemstabcat$ there is a natural weak
equivalence of spectra  
\[
\Map(\Umot(\Perf_{\bbS[t]}), \Umot(\cC)) \simeq \KEnd(\cC)\,.
\]
\end{theorem}

\begin{proof}
The first claim follows from the second.  We can verify the second
claim as follows.  Since by Lemma~\ref{cor:compact} the
$\infty$-category $\Perf_{\bbS[t]}$ is compact, it follows
from~\cite[7.13]{BGT} that we have a natural equivalence of
spectra 
$$\Map(\Umot(\Perf_{\bbS[t]},\Umot(\cA)) \simeq
K \idemfun(\Perf_{\bbS[t]},\cA)\, .$$

Now we apply Proposition~\ref{prop:endident}.  To complete the 
comparison, we need to consider the Waldhausen structures giving rise
to the algebraic $K$-theory of $\rep(S[t], \tilde{\aA})$ and
$\KEnd(\perf(\tilde{\cA}))$ --- Proposition~\ref{prop:agreement1}
allows us to use the Waldhausen models for the $K$-theory spectra.

The Waldhausen structure on $\End(\tilde{\cA})$ is inherited from the
forgetful functor $\End(\perf(\tilde{\cA})) \to \perf(\tilde{\cA})$;
cofibrations are maps which are cofibrations in $\perf(\tilde{\cA})$.
On the other hand, the Waldhausen structure on $\rep(S[t],\cA)$ is
given by maps which are cofibrations of bimodules.  The identity
functor $\rep(S[t],\tilde{\cA}) \to \End(\perf(\tilde{\cA}))$ is
exact, and evidently induces an equivalence on homotopy categories.
Since in both Waldhausen structures all maps are weak cofibrations and
we have functorial factorization, the generalized version of the
approximation theorem (e.g., see~\cite[1.1]{BM}) implies that this
exact functor induces an equivalence on $K$-theory spectra.
\end{proof}

\subsection{Agreement property}\label{s:agreement}
We start by showing that definition~\ref{def:K-th-endo} subsumes the
setting of Waldhausen categories. 

\begin{proposition}{(Agreement I)}\label{prop:agreement1}
Let $\aC$ be a DHKS-saturated Waldhausen category that admits functorial
factorization and $\N\aC[W^{-1}]$ the associated $\infty$-category.
Then there is a natural zig-zag of weak equivalences of spectra
between $\KEnd(\aC)$ and $\KEnd(\N\aC[W^{-1}])$.
\end{proposition}

\begin{proof}
First, assume that $\aC$ is a full subcategory of the cofibrant
objects in a combinatorial pointed model category $\aA$.
Then~\cite[1.3.4.25]{HA} implies that there is an equivalence 
\[
\N(\Fun(\underline{\mathbb{N}}, \aC)[W^{-1}] \htp \Fun(\Delta^1
/ \partial \Delta^1, \N \aC[W^{-1}]) 
\]
and therefore an equivalence of spectra
\begin{equation}\label{eq:newequiv}
K(\N(\Fun(\underline{\mathbb{N}},\aC)[W^{-1}]) \htp
K \End(\N \aC[W^{-1}]). 
\end{equation}

Although the Waldhausen structure on $\End(\aC)$ need not arise as
induced from a model structure by restriction to the cofibrant
objects, since $\End(\aC)$ inherits functorial factorization from
$\aC$ we can reduce to this case~\cite[7.7]{BGT}.  Next,
using~\cite[7.10]{BGT}, which compares the $K$-theory of a Waldhausen
category to the $K$-theory of the underlying $\i$-category, we find
that the left-hand side of \eqref{eq:newequiv} is equivalent to the
Waldhausen $K$-theory $\KEnd(\aC)$.  When $\aC$ is an arbitrary
Waldhausen category that admits functorial factorization and is
DHKS-saturated, we can again use~\cite[7.10]{BGT} to reduce to the
case of a full subcategory of the cofibrant objects in a pointed model
category.
\end{proof}

The following result relates the classical definition of the
$K$-theory of endomorphisms of an ordinary ring $A$ with the
definition given herein for the associated Eilenberg-Mac Lane spectrum 
$HA$.

\begin{theorem}{(Agreement II)}\label{thm:agreement2}
Let $A$ be an ordinary ring.  There exists a canonical zig-zag of weak
equivalences of spectra between the $K$-theory spectrum
$\KEnd(\mathbf{P}(A))$ of the exact category $\End(\mathbf{P}(A))$ and
the $K$-theory spectrum $\KEnd(\Perf_{HA})$ of the $\i$-category
$\End(\Perf_{HA})$. 
\end{theorem}

\begin{proof}
Using Proposition~\ref{prop:agreement1}, we have a zig-zag of weak
equivalences between $\KEnd(\Perf_{HA})$ and $\KEnd(\perf(HA))$.  We
can assume without loss of generality that $HA$ is a cofibrant
$S$-algebra.  Recall from the discussion of Example~\ref{ex:freeassoc}
that $\End(\perf(HA))$ is equivalent to the category of
$HA \sma \bN_+$-modules that are perfect as $HA$-modules.  Moreover,
$HA \sma \bN_+ \cong HA[t]$.  By~\cite[\S IV.2.4]{EKMM}, the derived
category $\Ho(HA[t])$ of $HA[t]$-modules is equivalent to the derived
category $\Ho(A[t])$ of complexes of $A[t]$-modules.  More generally,
we have a commutative diagram
\[
\xymatrix{
\Ho(HA[t]) \ar[r]^{\htp} \ar[d] & \Ho(A[t]) \ar[d] \\
\Ho(HA) \ar[r]^{\htp} & \Ho(A), \\
}
\]
where the vertical maps are induced from the forgetful functors.  As a
consequence, we can conclude that there is an equivalence of homotopy
categories between $\End(\perf(HA))$ and the Waldhausen category
$\End(\perf(A))$ where $\perf(A)$ denotes the Waldhausen category of
perfect complexes of $A$-modules, where the cofibrations are the split
monomorphisms and the weak equivalences the quasi-isomorphisms.

Now, using the identification (where $\Ch^b$ stands for bounded complexes)
\[
\End(\Ch^b(\mathbf{P}(A))) \cong \Ch^b(\End(\mathbf{P}(A))),
\]
we obtain a zig-zag of exact functors connecting $\End(\mathbf{P}(A))$
to $\End(\Ch^b(\mathbf{P}(A))$.  Specifically, let $\aE'$ denote the
Waldhausen category structure on $\Ch^b(\End(\mathbf{P}(A)))$ where the
cofibrations are the levelwise admissible monomorphisms and the weak
equivalences are the quasi-isomorphisms.  Then by~\cite[1.11.7]{TT},
the zig-zag of exact functors 
\[
\xymatrix{
\End(\mathbf{P}(A)) \ar[r] & \aE' & \Ch^b(\End(\mathbf{P}(A))) \ar[l]_-{\Id} \\
}
\]
induces equivalences on passage to $K$-theory. Therefore, it suffices
to show that the exact inclusion functor 
\[
\iota \colon \End(\Ch^b(\mathbf{P}(A))) \to \End(\perf(A))
\]
induces an equivalence on $K$-theory spectra.  We do this using a
modern reformulation of~\cite[1.11.7]{TT} arising from Barwick's
theory of the $K$-theory of exact $\i$-categories~\cite{Barwickexact}.
Since Lemma~\ref{lem:t} below implies that $\Ho(\End(\perf(A)))$ has
a bounded $t$-structure, we can apply~\cite[5.6.1]{Barwickexact} to
compare $K(\End(\perf(A)))$ to the Quillen $K$-theory of the heart of
the $t$-structure regarded as an exact category.  As the heart is
precisely $\End(\mathbf{P}(A))$, the fact that this equivalence
factors through $\iota$ shows that $\iota$ is an equivalence as well.
\end{proof}

We now establish the existence of a bounded $t$-structure on
$\End(\perf(A))$.  Recall that a $t$-structure on a triangulated
category $\aC$ is determined by a pair of full subcategories
$\aC_{\leq 0}$ and $\aC_{\geq 0}$ such that:
\begin{enumerate}
\item For objects $X,Y$ in $\aC$ such that $X \in \aC_{\geq 0}$ and
$Y \in \aC_{\leq 0}$, $\Hom_{\aC}(X,Y[-1]) = 0$.
\item There are inclusions $\aC_{\geq 0}[1] \subseteq \aC_{\geq 0}$
and $\aC_{\leq 0}[-1] \subseteq \aC_{\leq 0}$.
\item For any object $X$ in $\aC$, there exists a distinguished
triangle $X' \to X \to X''$ such that $X' \in \aC_{\geq 0}$ and
$X'' \in \aC_{\leq 0}[-1]$.
\end{enumerate}

A standard example of a $t$-structure is given by considering
$\Ho(\perf(A))$ for an ordinary ring $A$ (the derived category of
perfect complexes), and defining $\aC_{\geq 0}$ and $\aC_{\leq 0}$ to
be the complexes with non-negative homology and non-positive homology,
respectively.

We write $\aC_{\leq n}$ for $\aC_{\leq 0}[n]$ and $\aC_{\geq m}$ for
$\aC_{\geq 0}[m]$.  A $t$-structure is bounded if all objects in $\aC$
are contained in $\aC_{\leq n} \cap \aC_{\geq m}$ for some $n > m$.
The standard $t$-structure on $\Ho(\perf(A))$ is evidently bounded.

\begin{lemma}\label{lem:t}
Let $A$ be an ordinary ring.  Then the category
\[
\Ho(\End(\perf(A))) \htp \Ho(\End(\Perf_{HA}))
\]
is triangulated and 
has a bounded $t$-structure induced from the triangulation and
standard bounded $t$-structure on $\Ho(\perf(A))$. 
\end{lemma}

\begin{proof}
Recall that the homotopy category of a stable $\i$-category is
triangulated~\cite[1.1.2.14]{HA}.  Since $\End(\Perf_{HA})$ inherits
the structure of a stable $\i$-category from
$\Perf_{HA}$~\cite[1.1.3.1]{HA}, $\Ho(\End(\Perf_{HA}))$ has a
triangulation in which the distinguished triangles in
$\Ho(\End(\Perf_{HA}))$ are induced by the forgetful 
functor 
\[
\Ho(\End(\Perf_{HA})) \to \Ho(\Perf_{HA}).
\]
Since the underlying stable $\i$-category of $\perf(A)$ is equivalent
to $\Perf_{HA}$, this triangulation on $\Perf_{HA}$ can be regarded as
coming from the triangulation on $\perf(A)$.

The $t$-structure on $\Ho(\perf(A))$ induces one on
$\Ho(\End(\perf(A)))$ using full subcategories
$\Ho(\End(\perf(A)))_{\geq 0}$ and $\Ho(\End(\perf(A)))_{\leq 0}$
determined by the forgetful functor.  The only nontrivial condition to
check is that any $X$ can be fit into a triangle $X' \to X \to X''$
where $X'$ has nonnegative homology and $X''$ has nonpositive
homology.  For any connective ring spectrum $R$, there is a functorial
construction of the connective cover on the category of $R$-modules,
such that we have a natural transformation
$C \to \id$~\cite[4.2]{MandellShipley}.  Since it is functorial, this
construction passes to $\End(\perf(A))$, and the associated cofiber
sequence gives the required triangle.

Finally, it is clear that the induced $t$-structure on
$\Ho(\End(\perf(A)))$ is bounded since the one on $\Ho(\perf(A))$ is.
\end{proof}

We conclude the section with a technical lemma which gives a partial
analysis of the homotopy category of
$\mathrm{End}(\Ch^b(\mathbf{P}(A)))$.

\begin{lemma}\label{lem:esurj}
The induced functor
\begin{equation}\label{eq:Qiso}
\iota\colon \mathrm{End}(\Ch^b(\mathbf{P}(A))) \too \mathrm{End}(\perf(A))\,,
\end{equation}
is homotopically essentially surjective, where we equip each side
with the weak equivalences given by the underlying
quasi-isomorphisms.
\end{lemma}

\begin{proof}
We will decorate quasi-isomorphisms with the symbol $\sim$. Recall
that we have fully-faithful inclusions 
$\Ch^b(\mathbf{P}(A)) \hookrightarrow \mathrm{perf}(A) \hookrightarrow
\Ch(A)$ and that $\Ch(A)$ carries a projective Quillen model structure
in which every object is fibrant~\cite[2.3.11]{Hovey}. The associated
homotopy category is the derived category $\cD(A)$ of $A$. 

Let $(M^\bullet,\alpha)$ be an object of
$\mathrm{End}(\mathrm{perf}(A))$, \ie a complex
$M^\bullet \in \mathrm{perf}(A)$ and an endomorphism $\alpha$ of
$M^\bullet$.  Since $M^\bullet \in \mathrm{perf}(A)$, $M^\bullet$ has
the homotopy type of a wedge summand of a finite cell complex of
$A$-modules; that is, there exists an isomorphism in
$\cD(A)$ between $M^\bullet$ and a complex $P^\bullet \in
\Ch^b(\mathbf{P}(A))$. All the objects of $\Ch^b(\mathbf{P}(A))$ are
cofibrant~\cite[2.3.6]{Hovey} and so this isomorphism lifts
to a quasi-isomorphism $\theta: P^\bullet \stackrel{\sim}{\to}
M^\bullet$. Associated to $\alpha$ we obtain then a well-defined
endomorphism of $P^\bullet$ in the derived category $\cD(A)$. Since
$P^\bullet$ is cofibrant we can choose a representative
$\overline{\alpha}:P^\bullet \to P^\bullet$ of this endomorphism. We
obtain then a square  

\begin{equation}\label{eq:hsquare}
\xymatrix{
P^\bullet \ar[d]_{\overline{\alpha}} \ar[r]^-{\theta}_\sim & M^\bullet \ar[d]^\alpha \\
P^\bullet \ar[r]_\theta^\sim & M^\bullet 
}
\end{equation}
which is commutative only in the derived category $\cD(A)$. The proof will consist now on replacing the quasi-isomorphism $\theta$ in \eqref{eq:hsquare} by a zig-zag of strictly commutative squares relating $\overline{\alpha}$ to $\alpha$. Let
$$ 
\xymatrix@C=2em@R=1em{*+<1pc>{P^\bullet} \ar@{>->}[r]^-\sim & X^\bullet \ar@{->>}[r]^-\sim & M^\bullet
}
$$
be a factorization of $\theta$ provided by the projective model structure. Note that $X^\bullet$ is cofibrant since this is the case of $P^\bullet$. Moreover, it belongs to $\mathrm{perf}(A)$ since it is quasi-isomorphic to $M^\bullet$. The lifting properties of the projective model structure furnish us morphisms $\overline{\beta}$ and $\beta$ making the following two diagrams commutative
\begin{equation}\label{eq:2squares}
\xymatrix{
*+<1pc>{P^\bullet} \ar[d]_{\overline{\alpha}} \ar@{>->}[r]^-\sim & X_\bullet \ar[d]^{\overline{\beta}} && X_\bullet \ar[d]_{\beta} \ar@{->>}[r]^-\sim & M^\bullet \ar[d]^-\alpha \\
*+<1pc>{P^\bullet} \ar@{>->}[r]_-\sim & X^\bullet && X^\bullet \ar@{->>}[r]_-\sim & M^\bullet\,.
}
\end{equation}
By combining the squares \eqref{eq:hsquare}-\eqref{eq:2squares}, we conclude then that the endomorphisms $\overline{\beta}$ and $\beta$ of $X^\bullet$ agree in the derived category $\cD(A)$. Since $X^\bullet$ is cofibrant object there exists a cylinder object
$$
\xymatrix@C=1,5em@R=1em{
X^\bullet \oplus X^\bullet \ar@{>->}[rr]^-{[i_0 \,\,i_1]} && \mathrm{Cyl}(X^\bullet) \ar@{->>}[r]^-\sim &X^\bullet
}
$$
and a morphism $H$ making the following diagram commute
$$
\xymatrix{
X^\bullet \ar[r]^-{i_0}_-\sim \ar[dr]_{\overline{\beta}} & \mathrm{Cyl}(X^\bullet) \ar[d]^H & X^\bullet \ar[l]_-{i_1}^-\sim \ar[dl]^{\beta} \\
& X^\bullet & \,.
}
$$
Note that since $\mathrm{Cyl}(X^\bullet)$ is quasi-isomorphic to $X$, it also belongs to $\mathrm{perf}(A)$. Consider the following commutative solid diagram:
\begin{equation}\label{eq:new}
\xymatrix{
*+<1pc>{X^\bullet \oplus X^\bullet} \ar[rr]^{\overline{\beta}\oplus \beta}\ar@{>->}[d]_-{[i_0\,\,i_1]} && X^\bullet \oplus  X^\bullet \ar[rr]^{[i_0\,\, i_1]} && \mathrm{Cyl}(X^\bullet) \ar@{->>}[d]^\sim \\
\mathrm{Cyl}(X^\bullet) \ar[rrrr]_-H \ar@{-->}[urrrr]^-{\widetilde{H}} &&&& X^\bullet\,.
}
\end{equation}
By the lifting properties of the projective model structure there exists a well-defined morphism $\widetilde{H}$ as above making both triangles of the diagram commute. Now, consider the following commutative diagram:
\begin{equation}\label{eq:diag-final}
\xymatrix{
*+<1pc>{P^\bullet} \ar[d]_-{\overline{\alpha}} \ar@{>->}[r]^-\sim & X^\bullet \ar[d]^{\overline{\beta}} \ar[r]^-{i_0}_-\sim & \mathrm{Cyl}(X) \ar[d]^-{\widetilde{H}} & X^\bullet \ar[l]_-{i_1}^-\sim \ar[d]^{\beta} \ar@{->>}[r]^{\sim} & M^\bullet \ar[d]^\alpha \\
*+<1pc>{P^\bullet} \ar@{>->}[r]_-\sim & X^\bullet \ar[r]^-\sim_-{i_0} & \mathrm{Cyl}(X) & X^\bullet \ar[l]^-{i_1}_-\sim \ar@{->>}[r]_\sim & M^\bullet\,,
}
\end{equation}
Note that the commutativity of the two interior squares is equivalent
to the commutativity of the upper triangle in \eqref{eq:new}. The
diagram \eqref{eq:diag-final} can then be interpreted as a zig-zag of
quasi-isomorphisms in the category $\mathrm{End}(\mathrm{perf}(A))$
relating $(P^\bullet,\overline{\alpha})$ with
$(M^\bullet, \alpha)$. As a consequence,
$(P^\bullet,\overline{\alpha})$ and $(M^\bullet, \alpha)$ become
isomorphic in the homotopy category
$\Ho(\mathrm{End}(\Ch^b(\mathbf{P}(A))))$. Since $(P^\bullet,\alpha)$
belongs to $\mathrm{End}(\Ch^b(\mathbf{P}(A)))$ the proof is then
finished. 
\end{proof}

\section{Endomorphisms of projective modules}\label{sec:endomorphisms}

In this section, we compute $\pi_0 \KEnd(\Perf_{R})$ for a connective
ring spectrum $R$; see theorem~\ref{thm:endcomp}.  To do this, we
begin by rapidly reviewing the theory of free and projective modules
over connective ring spectra; see \cite[8.2.2.4]{HA} for further
details or \cite[\S 2]{AG} for an exposition of the relevant results.
Throughout this section, $R$ will stand for a {\em connective} ring
spectrum $R$.

\begin{definition}\label{def:free}
An $R$-module $M$ is {\em (finite) free} there exists a (finite) set
$I$ and an equivalence of $R$-modules $R^{\oplus I}\simeq M$.  An
$R$-module $M$ is {\em (finite) shifted free} if there exists a 
(finite) $\bZ$-graded set $I\to\bZ$ and an $R$-module equivalence
$\bigoplus_{n\in\bZ} \Sigma^n R^{\oplus I_n}\simeq M$, where $I_n$
denotes the fiber of $I\to\bZ$ over $n\in\bZ$.
\end{definition}

\begin{definition}\label{def:proj}
An $R$-module $P$ is said to be {\em projective} if it is projective
as an object of the $\i$-category $\Mod_R^{\geq 0}$ of connective
$R$-modules, i.e., the functor
\[
\map_R(P,-)\colon \Mod_R^{\geq 0}\too\aT
\]
commutes with geometric realizations of connective $R$-modules.
\end{definition}

\begin{remark}
We could give Definition~\ref{def:free} for modules over
non-connective ring spectra, but the notion of shifted free module is
really only sensible if $R$ is not periodic; that is, if the only
integer $n$ for which there exists an equivalence $\Sigma^n R\to R$ is
zero (of course, this property always holds for connective ring
spectra).  Similarly, Definition~\ref{def:proj} implies that $P$ is a
connective $R$-module to start with; again, this notion of projective
$R$-module is only properly behaved when $R$ is connective.  This is
because, for an arbitrary ring spectrum $R$, connective or not, there
are no nontrivial projective objects of the $\i$-category $\Mod_R$ of
all right $R$-modules (and of course the same is true in the
$\i$-category $\Mod_{R^{\op}}$ of left $R$-modules). See the argument
immediately following \cite[2.5]{AG} for details.
\end{remark}

\begin{proposition}\label{prop:proj}
Suppose that $P$ is a connective $R$-module.  Then the following are
equivalent: 
\begin{enumerate}
\item
The $R$-module $P$ is projective.
\item
The $R$-module $P$ is a retract of a free $R$-module.
\item
The functor
\[
\map_R(P,-) \colon \Mod_R^{\geq 0}\too\aT
\]
preserves surjections (i.e., morphisms which are surjective on $\pi_0$).
\item
Given a surjection (on $\pi_0$) of (not necessarily connective)
$R$-modules $N\to M$ and any map $P\to M$, there exists a map $g\colon
P\to N$ such that the resulting diagram 
\[
\xymatrix{
& P\ar[ld]\ar[rd] &\\
N\ar[rr] & & M}
\]
commutes in $\Mod_R$.
\end{enumerate}
\end{proposition}

\begin{proof}
See \cite[\S 8]{HA}.  Note that, in (4), we may assume without loss of
generality that $M$ and $N$ are connective, as
$\pi_0\map(P,M)\cong\pi_0\map(P,\tau_{\geq 0} M)$.
\end{proof}

Proposition~\ref{prop:proj} shows that projective modules over
connective ring spectra behave in much the same way as projective
modules over ordinary rings.  In particular, it motivates the next
definition.

\begin{definition}
An $R$-module $M$ is said to be {\em shifted projective} if it is a
retract of a shifted free $R$-module. 
\end{definition}

We will only be interested in finite (shifted) projective $R$-modules.

\begin{definition}
A (shifted) projective $R$-module $M$ is {\em finite (shifted)
projective} if it is a retract of a finite (shifted) free $R$-module. 
\end{definition}

Note that finite (shifted) projective modules are perfect, since
finite (shifted) free modules are perfect and retracts of compact
objects are still compact.  We write $\Proj_R\subset \Perf_{R}$ for
the full subcategory of $\Perf_{R}$ consisting of the finite projective
$R$-modules, and we write $\Proj^\Sigma_R\subset \Perf_{R}$ for the full
subcategory of $\Perf_{R}$ consisting of the finite shifted projective
$R$-modules.  The following comparison provides a description of
$\Proj_R$ in terms of the discrete ring $\pi_0 R$.

\begin{proposition}
Let $R$ be a connective ring spectrum.
Then
\[
\pi_0 \colon \mathrm{Ho}(\Proj_R)\too\Proj_{\pi_0 R}
\]
is an equivalence of categories.
\end{proposition}

\begin{proof}
This is standard. For a proof, see \cite[2.12]{AG}, for instance.
\end{proof}

The first step towards computing $\pi_0(\KEnd(\Perf_{R}))$ of this
section is to prove that, for $R$ a connective ring spectrum, there is
an isomorphism $K_0(\End(\Proj_R))\cong K_0(\End(\Perf_{R}))$.  Here
to define $K_0(\End(\Proj_R)) = \pi_0 \KEnd(\Proj_R))$ we will specify
the structure of an $\i$-category with cofibrations on
$\End(\Proj_R)$.  First, we specify the cofibrations on $\Proj_R$ as
the maps $P \to P \coprod Q$ such that the cofiber (in $\Perf_{R}$) is
$Q$.  It is straightforward to check that this definition satisfies
the conditions of Definition~\ref{defn:infcof}.  Then we define a map
to be a cofibration in $\End(\Proj_R)$ if its image under the
forgetful functor $\End(\Proj_R) \to \Proj_R$ is a cofibration.

We begin by recalling the analogous result in the setting without
endomorphisms. 

\begin{theorem}\label{thm:computation}
For any connective ring spectrum $R$, the map
\[
i\colon K_0(\Proj_R)\too K_0(\Perf_{R}),
\]
induced by the inclusion $\Proj_R\to \Perf_{R}$, is an isomorphism.
\end{theorem}

\begin{proof}
We define a map
\[
j\colon K_0(\Perf_{R})\too K_0(\Proj_R)
\]
and check that it is inverse to $i$.
The idea is to show that any generator $[M]$ of $K_0(\Perf_{R})$ is a
signed sum of generators $[P]$ of $K_0(\Proj_R)$, and the proof is by
induction on the length $l=b-a$ of the Tor-amplitude of $M$
(see \cite{SGA6} or \cite{TT} for a discussion of Tor-amplitude in the
setting of derived categories and \cite[\S 2.4]{AG} for the
analogous treatment in the setting of modules over a ring spectrum). 

To this end, let $M$ be a perfect $R$-module.
By \cite[2.13.1]{AG}, there exists integers $a\leq b$ such that $M$
has Tor-amplitude contained in the interval $[a,b]$.  If $b-a=0$, then
by \cite[2.13.6]{AG}, $M\simeq\Sigma^a P$ for some finite projective
$R$-module $P$, in which case 
\[
[M]=(-1)^a[P]\in K_0(\Perf_{R})
\]
and we define
\[
j([M]):=(-1)^a[P]\in K_0(\Proj_R).
\]
This preserves the $K_0$-relations amongst $R$-modules with
Tor-amplitude of length $0$, because any such comes from a (split)
cofiber sequence of the form 
\[
\Sigma^a P\too\Sigma^a (P\oplus Q)\too\Sigma^a Q
\]
for some finite projective $R$-modules $P$ and $Q$ and
\[
[P]-[P\oplus Q]+[Q]=0\in K_0(\Proj_R).
\]
Inductively, suppose that we have defined $j$ on the subcategory of
$\Perf_{R}$ consisting of those perfect $R$-modules with Tor-amplitude
contained in an interval $[a,b]$ with $b-a< l$ for some positive
integer $l$, and let $M$ be a perfect $R$-module with Tor-amplitude
contained in an interval $[a,b]$ with $b-a=l$.  Then
by \cite[2.13.7]{AG}, there is a finite projective $R$-module $P$ and
a cofiber sequence 
\[
\Sigma^a P\too M\too N
\]
such that the cofiber $N$ has Tor-amplitude contained in $[a+1,b]$.
Since $\Sigma^a P$ and $N$ are perfect with Tor-amplitude of length
$b-a+1<l$, we set 
\[
(-1)^a j([P])-j([M])+j([N])=0\in K_0(\Proj_R).
\]
In other words, killing bottom homotopy groups by mapping in shifted
projectives results in a length $l$ filtration 
\[
M\simeq M_0\too M_1\too\cdots\too M_l\too M_{l+1}\simeq 0
\]
in which the filtration quotients $M_{k+1}/M_k\simeq\Sigma^{a+k+1}
P_i$, $0\leq k\leq l$, are shifts of (possibly trivial) finite
projective $R$-modules $P_k$. 

A separate induction implies that the resulting map
\[
j\colon \bigoplus_{M\in\pi_0\Perf_R}\bZ\to K_0(\Proj_R)
\]
respects the relations in $K_0$, as any relation amongst perfect
$R$-modules of Tor-amplitude of length not more than $l$ comes from a
cofiber sequence 
\[
L\too M\too N,
\]
and we may proceed on the length $k\leq l$ of the Tor-amplitude of $L$.
Indeed, if $k=0$ then $L$ is a finite shifted projective, in which
case the well-definedness is clear.  If $k>0$, then there is a finite
projective $P$ and a cofiber sequence $\Sigma^a P\to L\to L'$ such
that $L'$ has Tor-amplitude of length strictly less than $k$.
Defining $M'$ as the pushout of $L'\leftarrow L\to M$, we obtain a map
of cofiber sequences 
\[
\xymatrix{
L\ar[r]\ar[d] & M\ar[r]\ar[d] & N\ar[d]\\
L'\ar[r] & M'\ar[r] & N}
\]
in which the right vertical map is an equivalence.
This gives the relation $[L]-[M]=[L']-[M']$ in $K_0(\Perf_R)$, so that
\[
j([L])-j([M])-j([N])=j([L'])-j([M'])+j([N])=0
\]
by inductive hypothesis.  Since any relation in $K_0(\Perf_R)$
necessarily only involves finitely many perfect $R$-modules and
therefore has Tor-amplitude contained in some finite interval $[a,b]$,
it follows that $j$ is a well-defined homomorphism of abelian groups. 

It remains to show that $i$ and $j$ are inverse isomorphisms.
It is clear that for any finite projective $R$-module $P$,
$j(i([P]))=[P]$; likewise, if $M$ is a perfect $R$-module with
Tor-amplitude contained in an interval $[a,b]$ of length $b-a=l$, then 
\[
i(j([M]))=(-1)^{a_0}i([P_0])+\cdots
+(-1)^{a_l}i([P_l])=[\Sigma^{a_0}P_0]+\cdots +[\Sigma^{a_l}P_l]=[M] 
\]
since, by construction, $M$ admits a finite filtration with quotients
$\Sigma^{a_k}P_k$.
\end{proof}

The argument above demonstrates that any perfect $R$-module admits a
finite descending filtration whose quotients are finite shifted
projectives.  A similar argument, which is the content of the
following theorem, shows that objects of $\End(\Perf_{R})$ admit finite 
descending filtrations whose quotients are shifts of objects of
$\End\Proj_R$. 

\begin{theorem}\label{thm:pizero1}
For any connective ring spectrum $R$, the map
\[
i\colon K_0(\End(\Proj_R))\too \pi_0 K_0(\End(\Perf_{R})),
\]
induced by the inclusion $\Proj_R\to \Perf_{R}$, is an isomorphism.
\end{theorem}

\begin{proof}
As before, we define an analogous map
\[
j\colon K_0(\End(\Perf_{R}))\too K_0(\End(\Proj_R))
\]
and check that it is inverse to $i$.
Again, the proof goes by induction on the length of the Tor-amplitude
of a given generator $[\alpha\colon M\to M]$ of $K_0(\End(\Perf_{R}))$.  If
$M$ has Tor-amplitude contained in $[a,b]$ with $a-b=0$, then
$M\simeq\Sigma^a P$ for some finite projective $R$-module $P$, and we
define 
\[
j([M\overset{\alpha}{\to} M])=(-1)^a[P\overset{e}{\to} P]
\]
where $e\simeq\Sigma^{-a}\alpha:P\to P$ is the shifted endomorphism.
This is well-defined because, given a map of cofiber sequences
\[
\xymatrix{
P\ar[d]^e\ar[r] & P\oplus Q\ar[d]^{\left(\begin{array}{cc} e & g\\ h & f\end{array}\right)}\ar[r] & Q\ar[d]^f\\
P\ar[r]  & P\oplus Q\ar[r] & Q\,,}
\]
we have that
\[
[P\overset{e}{\to} P]-[P\oplus Q\overset{\left(\begin{array}{cc} e & g\\ h & f\end{array}\right)}{\to} P\oplus Q]+[Q\overset{f}{\to} Q]=0\in K_0(\End(\Proj_R)).
\]
Inductively, suppose that $j$ is defined for all
$[M\overset{\alpha}{\to}M]$ such that $M$ has Tor-amplitude contained
in an interval $[a,b]$ of length not more than $l$.  Let $P$ be a
finite projective and $\Sigma^a P\to M\to N$ a cofiber sequence such
that $N$ has Tor-amplitude contained in $[a+1,b]$ (\cite[2.7]{AG}).
Taking homotopy, the resulting exact sequence 
\[
\pi_0(P)\to\pi_0(\Sigma^{-a} M)\to\pi_0(\Sigma^{-a} N)
\]
shows that $P\to\Sigma^{-a} M$ is surjective on $\pi_0$, since
$\pi_0(\Sigma^{-a} N)\cong 0$ by lemma \ref{lem:connective} below.  It
follows from proposition \ref{prop:proj} that there exists an
endomorphism $e$ of $P$ making the lower triangle in the diagram 
\[
\xymatrix{
P\ar[r]\ar[d]\ar[rd] & \Sigma^{-a}M\ar[d]\\
P\ar[r] & \Sigma^{-a}M}
\]
commutes; that is, there is a map from $\Sigma^a P\overset{\Sigma^a
e}{\to}\Sigma^a P$ to $M\overset{\alpha}{\to} M$ in $\End(\Perf_{R})$. 
Writing $\beta\colon N\to N$ for the cofiber of this map in $\End(\Perf_{R})$,
we see that $N$ is a perfect $R$-module with Tor-amplitude in the
interval $[a+1,b]$.  Repeating this process gives a filtration
\[
\xymatrix{
M=M_0\ar[r]\ar[d]^{\alpha_0} & M_1\ar[r]\ar[d]^{\alpha_1}  & \cdots\ar[r] & M_{l}\ar[r]\ar[d]^{\alpha_{l}} & M_{l+1}\simeq 0\ar[d]^{\alpha_{l+1}}\\
M=M_0\ar[r] & M_1\ar[r] & \cdots\ar[r] & M_l\ar[r] & M_{l+1}\simeq 0}
\]
of $M\overset{\alpha}{\to} M$ in $\End \Perf_{R}$ in which
$\alpha_0\simeq\alpha$ and $\alpha_{l+1}\simeq 0$, and the filtration
quotients are of the form 
\[
\xymatrix{
\Sigma^{a+k+1} P_k\ar[r]^{\Sigma^{a+k}e_k} & \Sigma^{a+k+1}P_k}
\]
for some $P_k\overset{e_k}{\to} P_k$ in $\End\Proj_R$, $0\leq k\leq l$.
The same argument as before shows that this is well-defined, and that
the resulting homomorpism $j\colon K_0(\End \Perf_{R})\to K_0(\End\Proj_R)$ is
inverse to $i$. 
\end{proof}

\begin{lemma}\label{lem:connective}
Let $R$ be a connective ring spectrum and let $M$ be a perfect $R$-module with Tor-amplitude contained in $[0,\infty]$.
Then $M$ is connective.
\end{lemma}

\begin{proof}
This follows from the convergent spectral sequence
\[
\Tor^{\pi_* R}_p(\pi_q M,\pi_0 R)\implies\pi_{p+q}(M\land_R H\pi_0 R);
\]
in particular, if $\pi_q M\neq 0$ for some $q<0$, then $\pi_q M\otimes_{\pi_* R}\pi_0 R\neq 0$ since $R$ is connective, giving nonzero classes in $\pi_q(M\land_R H\pi_0 R)$.
\end{proof}

Finally, we want to complete the computation by showing that
$K_0(\End\Proj_R)\cong K_0(\End\Proj_{\pi_0 R})$.  To do so, we
require a technical lemma about rigidifying homotopy commutative
triangles in an $\i$-category $\aC$.

\begin{lemma}\label{lem:2skel}
Let $\aC$ be an $\i$-category and let $K$ be a $2$-skeletal simplicial
set.  Then any diagram
\[
\sigma\colon K\too\N(\Ho(\aC)),
\]
lifts to a diagram $\tau\colon K\to\aC$ such that $\eta\circ\tau=\sigma$,
where $\eta:\aC\to\N(\Ho(\aC))$ denotes the unit of the localization
$\Ho\colon \Set_{\Delta}\to\Cat \colon \N$.  
\end{lemma}

\begin{proof}
First suppose $K=\Delta^2$ and that $\aC=\N(\aC')$ for some fibrant simplicial category $\aC'$ such that $\Ho(\aC')\cong\Ho(\aC)$.
By adjunction, a $2$-simplex of $\N(\aC')$ is a map $\tau:\mathfrak{C}[\Delta^2]\to\aC'$, which is to say objects $\tau_i$, $0\leq i\leq 2$, maps $\tau_{ji}:\tau_i\to \tau_j$, $0\leq i<j\leq 2$, and a homotopy $\tau_{210}:\tau_{20}\to \tau_{21}\circ \tau_{10}$ in $\map_{\aC'}(\tau_0,\tau_2)$.
Since we're given a map $\sigma:\mathfrak{C}[\Delta^2]\to\Ho(\aC')$, we have homotopy classes of maps $\sigma_{ji}\in\pi_0\map_{\aC}(\sigma_i,\sigma_j)$ such that $\sigma_{20}=\sigma_{21}\circ\sigma_{10}$.
Taking $\tau_i=\sigma_i$, we may choose representative $\tau_{ji}:\tau_i\to\tau_j$ of $\sigma_{ji}$, and as the two resutling maps $\tau_{20}$ and $\tau_{21}\circ\tau_{10}$ from $\tau_0$ to $\tau_2$ are homotopic, we may also choose a $1$-simplex $\tau_{210}:\Delta^1\to\map_{\aC'}(\tau_0,\tau_2)$ realizing this.

For the general case, note that there exists a categorical equivalence
\[
f:\N(\mathfrak{C}[\aC]^\mathrm{fib})\to\aC
\]
which we may suppose is an isomorphism on homotopy categories.
First lift the $1$-skeleton $\sk_1 K$ to $\N(\mathfrak{C}[\aC]^\mathrm{fib})$ by choosing representative for homotopy classes of arrows in $\aC$, and then extend this to the $2$-skeleton by choosing lifts of each $2$-simplex compatibly with the chosen lifts on the boundary.
Composing with $f$ then gives the desired lift to $\aC$.
\end{proof}

\begin{corollary}\label{cor:surj}
Let $\aC$ be an $\i$-category. Then the canonical map
\[
\Ho(\End(\aC))\too\End(\Ho(\aC))
\]
is surjective on equivalence classes of arrows.
\end{corollary}

\begin{proof}
Since $\Delta^1\times\Delta^1/\partial\Delta^1$ is $2$-skeletal, the map
\[
\pi_0\map(\Delta^1\times\Delta^1/\partial\Delta^1,\aC)\too\pi_0\map(\Delta^1\times\Delta^1/\partial\Delta^1,\N(\Ho(\aC)))
\]
is surjective by lemma \ref{lem:2skel}.
But, by adjunction, the source is isomorphic to the set of equivalence classes of arrows in $\Ho(\End(\aC))$, and the target is isomorphic to the set of equivalence classes of arrows in $\End(\Ho(\aC))$.
\end{proof}

\begin{proposition}\label{prop:hoendvsendho}
If $R$ is a connective ring spectrum, then the canonical functor
\[
i\colon K_0(\End\Proj_R))\too K_0(\End\Ho(\Proj_R))
\]
is an isomorphism.
\end{proposition}

\begin{proof}
Using the presentation for $K_0$ of \ref{eqn:K0}, we first observe that $i$ is surjective, as
\[
\End(\Proj_R)^\simeq\too\N(\End(\Ho(\Proj_R)))^\simeq
\]
is surjective on $\pi_0$.
To see that $i$ is also injective, we must show that any exact sequence in $\End\Ho(\Proj_R)$ lifts to an exact sequence in $\End\Proj_R$.
Since an exact sequence in $\End\Proj_R$ is in particular a cofiber
sequence, any exact sequence in $\iSdot[2](\End\Proj_R)$ is determined
(up to contractible ambiguity) by a suitable arrow
$\Delta^1\to\End\Proj_R$.  Thus, the vertical fibers in the commutative square
\[
\xymatrix{
\iSdot[2](\End\Proj_R)^\simeq\ar[r]\ar[d]
& \iSdot[2](\End\N\Ho(\Proj_R))^\simeq\ar[d]\\ 
\map(\Delta^1,\End\Proj_R)\ar[r] & \map(\Delta^1,\N\Ho(\Proj_R))\,,}
\]
in which the vertical maps are the restrictions along
$\Delta^{\{0,1\}}\to\Delta^2$, are contractible, so the diagram is
cartesian.  But the bottom horizontal map is surjective on $\pi_0$ by
corollary \ref{cor:surj}, so the top horizontal map must be surjective
on $\pi_0$ as well. 
\end{proof}

\begin{corollary}\label{cor:comparison}
The map
\[
\End\Ho(\Proj_R)\too\End\Proj_{\pi_0 R}
\]
induced by the equivalence $\pi_0 \colon \Ho(\Proj_R)\simeq\Proj_{\pi_0
R}$ is a $K_0$-isomorphism. 
\end{corollary}

\begin{proof}
This is immediate from proposition \ref{prop:hoendvsendho}.
\end{proof}

Finally, assembling the comparisons of Corollary~\ref{cor:comparison},
Proposition~\ref{prop:hoendvsendho}, and Theorem~\ref{thm:pizero1}, we
obtain the following:

\begin{theorem}\label{thm:endcomp}
For every connective ring spectrum $R$ one has an isomorphism
\[
K_0(\End(\Perf_{R})) \cong K_0(\End(\mathbf{P}(\pi_0(R)))) = K_0(\End(\pi_0 R))
\]
of abelian groups.
\end{theorem}

\section{Natural operations}\label{sec:operations}
In this section we make use of the theory of noncommutative motives to classify all the natural transformations of the above functor \eqref{functor:End}; see Theorem~\ref{thm:natural}.
The constructions of \S\ref{sub:AG} are functorial on exact categories and hence (after the usual fixes associated to the fact that the passage from rings to exact categories of modules is only a pseudo-functor, e.g., see~\cite[9.1]{BM}) gives rise to a well-defined functor
\begin{eqnarray}
\KEnd(\mathbf{P}(-))\colon \mathrm{Rings} \too \cS
\end{eqnarray}
from ordinary rings to symmetric spectra.  In particular, we have the
classical functor
\begin{eqnarray}\label{eq:functors}
K_0(\End(\mathbf{P}(-)))\colon \mathrm{Rings} \too \mathrm{Ab}
\end{eqnarray}
from ordinary rings to abelian groups.
As explained by Almkvist in \cite[page~339]{Almkvist}, a major problem
is the computation of all the natural transformations of the above
functor \eqref{eq:functors}; see
also \cite{Stienstra1}\cite[\S1]{Stienstra}. Classical examples are
given by the Frobenius operations $F_n \colon [(M,\alpha)] \mapsto
[(M,\alpha^n)], n \geq 0$, and by the Verschiebung operations $V_n \colon
[(M, \alpha)] \mapsto [(M^{\oplus n}, V_n(\alpha))]$, where 
\begin{equation}\label{eq:cver}
V_n(\alpha):=\left[
\begin{array}{cccccccccc}
0 & \cdots & \cdots & 0  & (-1)^{n+1} \alpha \\
1 & \ddots &  &  \vdots& 0 \\
0 & \ddots & \ddots &  \vdots & \vdots\\
\vdots &\ddots&\ddots&0 &\vdots \\
0 &\cdots &0 &1&0\\
\end{array}\right]_{(n \times n)} \,.
\end{equation}
These natural operations can be generalized to the $\infty$-categorical
setting as follows: 
\begin{definition}{(Frobenius)}\label{defn:frob}
Let $f_n$ be the endofunctor of $\underline{\bbN}$ induced by the monoid map $n:\bbN\to\bbN$ which sends $m$ to $nm$.
Using the equivalence
\[
\Fun(\N(\underline{\bbN}),\cA)\too\Fun(\Delta^1/\partial\Delta^1,\cA)
\]
induced by the categorical equivalence $\Delta^1/\partial\Delta^1\to\N(\underline{\bbN})$, one obtains by
pre-composition with $f_n$ an exact functor
$f_n^\ast \colon \End(\cA) \to \End(\cA)$ and consequently a map of
symmetric spectra $\KEnd(\cA) \to \KEnd(\cA)$. This construction is
natural on $\cA$ and hence gives rise to a natural transformation
$F_n\colon \KEnd \Rightarrow \KEnd$ of the
functor \eqref{functor:End} that we call the {\em
$n^{\mathrm{th}}$-Frobenius operation}.  
\end{definition}

For an $\i$-category $\aC \in \icat$, there is a natural functor
$\aC \to \End(\aC)$, induced by projection $\Delta^1\partial\Delta^1\to\Delta^0$, that takes $x$ to $\id:x\to x$.
By precomposing with the forgetful functor $\End(\aC) \to \aC$, we obtain the composite functor
\begin{eqnarray}
\iota \colon \End(\aC) \to \aC \to \End(\aC)
\end{eqnarray}
which sends the endomorphism $\alpha:x\to x$ to $\id:x\to x$.
Moreover, given an $\i$-category $\aD$ with finite coproducts, given
functors $f_i \colon \aC \to \aD$ we can construct a functor
$\coprod f_i \colon \aC \to \aD$ as the composite
\[
\xymatrix{
\aC \ar[r] & \prod_i \aC \ar[r]^-{\prod f_i} & \prod_i \aD \ar[r] & \aD,
}
\]
where the last map is a choice of functorial coproduct.  Similarly,
given $\tau \in \Sigma_n$, we can permute the factors of this coproduct
by $\tau$.  We now use these constructions to generalize the
Verschiebung:

\begin{definition}{(Verschiebung)}\label{defn:versh}
For each $\infty$-category $\cA \in \idemstabcat$ there is an
endofunctor on $\End(\aA)$ define by applying the cyclic permutation
to the coproduct of $(n-1)$ copies of $\iota$ and one copy of
$(-1)^{n+1} \id$.  This functor gives rise to a natural transformation 
$\End(\aA) \to \End(\aA)$ and hence a natural transformation 
$V_n \colon \KEnd \rightarrow \KEnd$ that we call the {\em
$n^{\mathrm{th}}$-Verschiebung operation}.
\end{definition}

Our solution to the problem stated by Almkvist if the following: let 
\begin{equation*}
W_0(\bbZ[t]):= \left\{\frac{1+p_1(t)r + \cdots + p_i(t)r^i  +\cdots +
p_n(t)r^n}{1+q_1(t)r+ \cdots + q_j(t)r^j + \cdots  + q_m(t)r^m} \,|\,
p_i(t),q_j(t) \in \bbZ[t] \right\} 
\end{equation*}
be the multiplicative group of fractions of polynomials in the
variable $r$ with coefficients in $\bbZ[t]$ and constant term $1$. 

\begin{theorem}\label{thm:natural}
There is a canonical weak equivalence of spectra between the spectrum
$ \Nat(\KEnd,\KEnd)$ of natural transformations of the
functor \eqref{functor:End} and the spectrum 
$\KEnd(\Perf_{\bbS[t]})$. In particular, the abelian
group $\pi_0\Nat(\KEnd,\KEnd)$ of natural transformations up to
homotopy is isomorphic to 
\begin{equation*}
K_0\mathrm{End}(\Perf_{\bbS[t]}) \simeq
K_0\mathrm{End}(\mathbf{P}(\pi_0 \bbS[t])) \simeq
K_0\mathrm{End}(\mathbf{P}(\bbZ[t])) \simeq \bbZ \oplus
W_0(\bbZ[t])\,. 
\end{equation*}
Moreover, under these isomorphisms, the Frobenius operations $F_n$ are
identified with the elements $(1,1+r^nt)$ and the Verschiebung
operations $V_n$ with the elements $(n,1+rt^n)$.
\end{theorem}

\begin{proof}
The natural equivalence of spectra
$\Nat(\KEnd,\KEnd) \htp \KEnd(\Perf_{\bbS[t]})$ follows from
Lemma~\ref{lem:aux} below (with $E$ the functor \eqref{functor:End}).
The isomorphisms follow from Theorem~\ref{thm:endcomp} applied to
$R=\bbS[t]$, from the equality $\pi_0 \bbS[t] =\bbZ[t]$, and from
Almkvist's isomorphism (see \cite{Almkvist1})
\begin{eqnarray}\label{eq:almk}
K_0\mathrm{End}(\mathbf{P}(A)\stackrel{\cong}{\to} K_0(A) \oplus
W_0(A) && (M,\alpha)] \mapsto ([M], \mathrm{det}(\Id + \alpha
r)) \nonumber
\end{eqnarray}
applied to $A = \mathbb{Z}[t]$. The identifications of $F_n$ and $V_n$
as the elements in question also follows from the preceding
computation and Theorem~\ref{thm:agreement2}; on $\pi_0$, the operations of
Definitions~\ref{defn:frob} and \ref{defn:versh} give rise to the
classical operations on the $K$-theory of endomorphisms.
Specifically, on passage to $K_0$, it is clear that the operation of
Definition~\ref{defn:frob} takes the the class $[M,\alpha]$ to
$[M,\alpha^n]$.  Moreover, since for any connective ring spectrum $R$
and compact $R$-module $M$, $\Map_R(\vee_n M, \vee_n
M) \htp \prod_n \coprod_n \Map_R(M,M)$, on passage to $K_0$ the
operation of Definition~\ref{defn:versh} gives rise to the matrix
specified above in equation~\eqref{eq:cver}.  Now
Theorem~\ref{thm:agreement2} coupled with Almkvist's identification of
these operations in terms of the isomorphism~\eqref{eq:almk}
establishes the desired comparison.
\end{proof}

\begin{lemma}\label{lem:aux}
For every additive invariant $E\colon \idemstabcat \to \cS_\infty$
there is a natural equivalence of spectra 
\begin{equation}\label{eq:equiv}
\Nat(\KEnd,E) \to E(\Perf_{\bbS[t]})\,.
\end{equation}
\end{lemma}
\begin{proof}
As shown in Theorem~\ref{thm:co-rep}(i), the functor
$\KEnd\colon \idemstabcat \to \cS_\infty$ is an additive
invariant.  Hence, by equivalence \eqref{eq:induced} one 
obtains well-defined colimit preserving functors
\[
\overline{\KEnd},\overline{E} \colon \Motadd \to \cS_{\infty}
\]
satisfying $\overline{\KEnd} \circ \Umot \simeq \KEnd$ and
$\overline{E} \circ \Umot \simeq E$, as well as a natural
equivalence of spectra 
\begin{equation}\label{eq:equivalence1}
\Nat(\KEnd,E) \to \Nat(\overline{\KEnd},\overline{E})\,.
\end{equation}
By Theorem~\ref{thm:co-rep}(ii) the functor $\overline{\KEnd}$ is
co-represented in $\Motadd$ by the noncommutative motive 
$\Umot(\Perf_{\bbS[t]})$.  Hence, the $\infty$-categorical version
of the Yoneda lemma~\cite[\S 5.1.3]{HTT} provides an equivalence
of spectra 
\begin{equation}\label{eq:equivalence2}
\Nat(\overline{\KEnd},\overline{E}) \too \overline{E}(\Umot(\Perf_{\bbS[t]}) = E(\Perf_{\bbS[t]})\,.
\end{equation} 
By combining \eqref{eq:equivalence1} with \eqref{eq:equivalence2} we
obtain then the above equivalence \eqref{eq:equiv}. 
\end{proof}

\section{(Rational) Witt vectors}\label{sec:Witt}

Witt vectors were introduced in the thirties by
E.~Witt~\cite{Witt}. Given a commutative ring $A$, the {\em Witt ring}
$W(A)$ of $A$ is the abelian group of all power series of the form
$1+a_1r + a_2r^2 + \cdots$, with $a_i \in A$, endowed with the
multiplication $\ast$ determined by the equality
$(1-a_1r)\ast(1-a_2r)=(1-a_1a_2r)$. The {\em rational} Witt ring
$W_0(A)$ of $A$ consists of the elements which are quotients of polynomials; that is, those of the form 
$$ \left\{\frac{1+a_1r + \cdots + a_ir^i  +\cdots +
a_mr^m}{1+b_1r+ \cdots + b_jr^j + \cdots  + b_nr^n} \,|\, a_i,b_j \in
A\right\}\subset W(A);$$ 
consult~\cite{Hazewinkel} for further details.

Recall from \cite[\S2.3]{BGT}
that the category $\idemstabcat$ carries a symmetric monoidal
structure in which the tensor product $-\otimes^\vee-$ is
characterized by the property that functors out of
$\cA \idemtimes \cB$ are in correspondence with functors out of the
product $\cA \times \cB$ which preserve colimits in each variable.  The
$\idemtimes$-unit is the $\infty$-category $\Perf_{\bbS}$. 

\begin{proposition}\label{prop:coalgebra}
Let $M$ be a monoid in the $\i$-category of spaces. Then the
$\i$-category $\Perf_{\bbS[M]}$ of perfect modules for the monoid-ring
$\bbS[M]$ carries a canonical counital, coassociative, and cocommutative
coalgebra structure in $\idemstabcat$.  
\end{proposition}

\begin{proof}
First recall that, if $\aC^\otimes$ is a symmetric monoidal
$\i$-category in which the tensor product is the coproduct, then the
projection $\CAlg(\aC^\otimes)\to\aC$ is an equivalence.  Taking $\aC$
to be the opposite of the $\i$-category of $A_\i$-spaces, we see that
$M$ has a coassociative and cocommutative coalgebra structure, so that
$\bbS[M]\simeq\Sigma^\infty_+M$ is a coassociative and cocommutative
object in $A_\i$-spectra.  Now, we note that the ``one-object spectral
category'' functor $\Alg_\bbS\to\N(\Cat_\aS)[W^{-1}]$ extends to a
symmetric monoidal $\Alg_\bbS^\otimes\to\N(\Cat_\aS)[W^{-1}]^\otimes$,
and according to \cite[\S3]{BGT1}, $\Cat_\i^{\mathrm{perf}\otimes}$ is
a symmetric monoidal localization of $\N(\Cat_\aS)[W^{-1}]^\otimes$.
Thus $\Perf_{\bbS[M]}$ inherits a canonical coassociative and
cocommutative coalgebra structure.
\end{proof}

Specializing to the case in which $M=\bbN$, the free $A_\i$-monoid on one generator, this amounts to saying that the diagonal $\Delta \colon \bbN \to \bbN \times \bbN$ and the projection $\bbN\to *$ are monoid maps.
Applying $\Sigma^\infty_+$, we obtain ring maps
\begin{eqnarray*}
\Delta\colon \bbS[t] \stackrel{t \mapsto t\wedge t}{\to} \bbS[t] \wedge \bbS[t] && \epsilon\colon \bbS[t] \stackrel{t=1}{\to} \bbS.
\end{eqnarray*}
As proved in \cite[\S4]{BGT1}, the category of noncommutative motives
$\Motadd$ carries a symmetric monoidal structure making the universal
additive invariant $\Umot\colon \idemstabcat \to \Motadd$ symmetric
monoidal. Hence, by Proposition \ref{prop:coalgebra}, the
noncommutative motive $\Umot(\Perf_{\bbS[t]})$ becomes a
counital coassociative coalgebra in $\Motadd$.  As a consequence, we
obtain the following result:

\begin{proposition}\label{prop:maplax}
The functor 
\[
\Map(\Umot(\Perf_{\bbS[t]}),-) \colon \Motadd \to \ispec 
\]
is lax symmetric monoidal.
\end{proposition}

\begin{proof}
More generally, we observe that if $\aC^{\otimes}$ is a symmetric
monoidal stable $\i$-category, the bifunctor $\aC^{\op} \times \aC \to \ispec$
specified as the adjoint of the spectral Yoneda embedding lifts to a
lax symmetric monoidal bifunctor
$(\aC^{\op}\times \aC)^{\otimes} \to \ispec^{\otimes}$. This follows
from the fact that the Yoneda embedding preserves limits and is
symmetric monoidal~\cite[6.3.1.12]{HA} by the relative
adjoint functor theorem~\cite[7.3.2.7]{HA}.  Specializing to the case
of $\Motadd$, this implies in particular that the functor
$\aC \to \ispec$ induced by restricting to a commutative algebra in
$\Motadd^{\op}$ (i.e., a counital coassociative coalgebra) is lax
symmetric monoidal.
\end{proof}

\begin{theorem}\label{thm:decomposition}
The ring maps $\bbS\stackrel{\iota}{\to}\bbS[t]$ and
$\bbS[t] \stackrel{t=0}{\to} \bbS$ give rise to a wedge sum
decomposition
$\Umot(\Perf_{\bbS[t]}) \simeq \Umot(\Perf_{\bbS}) \vee \mathbb{W}_0$
of counital coassociative cocommutative coalgebras in
$\Motadd$. Moreover, for every ordinary commutative ring $A$, we have
an isomorphism of commutative rings 
\begin{equation}\label{eq:Witt2}
\pi_0\Map(\mathbb{W}_0,\Umot(\Perf_{HA}))\simeq W_0(A)\,.
\end{equation}
\end{theorem}

\begin{proof}
Recall from Proposition~\ref{prop:coalgebra} that
$\bbS[t]\simeq\bbS[\bbN]$ and $\bbS\simeq\bbS[*]$ are counital,
coassociative, and cocommutative coalgebras.  It will be convenient to
briefly work with a point-set model of these coalgebras.  We can do
this using the suspension spectra $\Sigma^{\infty}_+ \bbN$ and
$\Sigma^{\infty}_+ *$ in the category of symmetric spectra; these
constructions represent coalgebras on the $\i$-categorical level
because $\Sigma^{\infty}_+ \bbN$ is cofibrant.  Now, note that the
diagrams    
$$
\xymatrix@-1pc{
\bbS \wedge \bbS \ar[r]^-{\iota \wedge \iota} & \bbS[t] \wedge \bbS[t] & \bbS[t] \wedge \bbS[t] \ar[rr]^-{(t=0)\wedge (t=0)} && \bbS \wedge \bbS & \bbS \ar@{=}[r] \ar@{=}[d] & \bbS \\
\bbS \ar[u]^-{1 \mapsto 1 \wedge 1} \ar[r]_-\iota & \bbS[t] \ar[u]_-{t \mapsto t \wedge t} & \bbS[t]\ar[rr]_-{(t=0)} \ar[u]^-{t \mapsto t \wedge t} && \bbS \ar[u]_-{1 \mapsto 1 \wedge 1} & \bbS \ar@{=}[u] \ar[r]_\iota & \bbS[t] \ar[u]_-{(t=1)}
}
$$
commute in the category of symmetric spectra.
Hence $\iota$ and $(t=0)$ descend to coalgebra maps with $\iota$
counit preserving.  Since $(t=0) \circ \iota =\Id$ the composition
$\iota \circ (t=0)$ is an idempotent of $\bbS[t]$.  For the remainder
of the proof we work with this idempotent in the $\i$-categorical
setting. 

We can realize this idempotent in terms of the composite maps
$* \to \mathbb{N} \to *$ and $\mathbb{N} \to * \to \mathbb{N}$ ---
these induce a retraction of coalegbras in $A_\infty$ monoids.
Applying $\Sigma^\infty_+$ and the composite functor
\begin{equation}\label{eq:composition}
\istabcat \stackrel{(-)^{\perf}}{\to} \idemstabcat \stackrel{\Umot}{\to} \Motadd\, ,
\end{equation}
we obtain then a retraction of coalgebra objects 
$$\Umot(\Perf_{\bbS}) \to \Umot(\Perf_{\bbS[t]}) \to \Umot(\Perf_{\bbS})\,.$$
Since $\Motadd$ is idempotent complete, we may split an idempotent
endomorphism $e\colon M\to M$ in $\Motadd$ by taking the filtered colimit 
\[
M_0\simeq\colim\{M\overset{e}{\too} M\overset{e}{\too}
M\overset{e}{\too}\cdots\}; 
\] 
taking the fiber $M_1\to M\to M_0$, we obtain a splitting $M\simeq
M_0\vee M_1$ of $M$, where $e$ restricts to the identity on $M_0$ and
zero on $M_1$.  Lastly, we note that colimits in the $\i$-category of
coalgebra objects are computed in the underlying $\i$-category.
Putting all of this together, we see that we can decompose
$\Umot(\Perf_{\bbS[t]})$ as a coproduct of $\Umot(\Perf_{\bbS})$
together with the coassociative and cocommutative coalgebra
object $\mathbb{W}_0$ of $\Motadd$.  To see that $\mathbb{W}_0$ is
counital, we observe that by~\cite[\S 5.2.3]{HA} it suffices to
produce a homotopy counit; the existence of such now follows from
proposition~\ref{prop:appendix} (more generally the arguments of the
appendix provide a splitting on the level of the homotopy category).

The identification of $\pi_0\Map(\mathbb{W}_0,\Umot(\Perf_{HA}))\simeq
W_0(A)$ follows from the same considerations as in the argument for
Theorem~\ref{thm:natural}.  Specifically, we can identify
$\pi_0 \Map(\Perf_{\bbS[t]}, \Umot(\Perf_{HA})) \cong
K_0\End(\Perf_{HA})$ as $K_0(A) \oplus W_0(A)$, and Almkvist's
results~\cite[pages 2-3]{Almkvist1} imply that the maps that split off
$\mathbb{W}_0$ in the preceding argument split off the $W_0(A)$
component on $\pi_0$.  This splitting induces the stated commutative
ring isomorphism since it is induced by the splitting
$\Umot(\Perf_{\bbS[t]}) \simeq \Umot(\Perf_{\bbS}) \vee \mathbb{W}_0$
of counital coassociative cocommutative coalgebras.
\end{proof}

Isomorphism~\eqref{eq:Witt2} motivates the following definition:

\begin{definition}
The spectrum of rational Witt vectors of a $\infty$-category
$\cA \in \idemstabcat$ is defined as $\Map(\mathbb{W}_0,\Umot(\cA))$.
\end{definition}

The argument for Proposition~\ref{prop:maplax} and the fact that $\Umot$
is monoidal~\cite[\S4]{BGT1} yields the following corollary:

\begin{corollary}\label{cor:wittlax}
The mapping spectrum $\Map(\mathbb{W}_0, -)$ provides a lax symmetric
monoidal functor $\Motadd \to \ispec$.  Therefore, when $\cA$ is an $E_n$
object in $\idemstabcat$, the spectrum of rational Witt vectors
$\Map(\mathbb{W}_0,\cA)$ is an $E_n$ object in symmetric spectra. 
\end{corollary}

Specializing to the case of $E_n$ ring spectra, we have the following
further corollary.

\begin{corollary}\label{cor:last}
Let $R$ be an $E_n$ ring spectrum. Then the associated rational Witt
ring spectrum $\Map(\mathbb{W}_0,\Umot(\Perf_{R}))$ is an $E_{n-1}$
ring spectrum. 
\end{corollary}

\begin{proof}
By Lurie's resolution of Mandell's conjecture
(see \cite[8.1.2.6]{HA}), the $\i$-category of modules for an $E_n$
ring spectrum $R$ is an $E_{n-1}$ object in the $\i$-category
$\mathrm{Pr^L_{St}}$ of presentable stable $\i$-categories.  Since the
category of $R$-modules  is compactly generated and the symmetric
monoidal structure on the $\i$-category $\mathrm{Pr^L_{St,\omega}}$ of
compactly-generated stable $\i$-categories is induced by the symmetric
monoidal structure on $\mathrm{Pr^L_{St}}$, we can conclude
from~\cite[\S 5.5.7]{HTT} that the $\i$-category of perfect modules
for an $E_n$ ring spectrum is an $E_{n-1}$ object in $\idemstabcat$.
The result now follows from Corollary~\ref{cor:wittlax}.
\end{proof}

\appendix
\section{Splitting coalgebras}
In this appendix we verify some technical results about splitting of
(point-set) coalgebras.  Let $(\cC,\otimes, {\bf 1})$ be a symmetric monoidal category. Recall that a {\em coalgebra} $(A,\mu_A,\eta_A)$ in $\cC$ consists of an object $A \in \cC$ and two maps $\mu_A\colon A \to A\otimes A $ (the comultiplication) and $\eta_A\colon A \to {\bf 1}$ (the counit). If the diagram
$$
\xymatrix{
A \ar[d]_-{\mu_A} \ar[r]^-{\mu_A} & A \otimes A \ar[r]^-{\mu_A \otimes \Id} & (A \otimes A)\otimes A \ar[d]^-{\simeq} \\
A\otimes A \ar[rr]_{\Id \otimes \mu_A} && A\otimes (A \otimes A) 
}
$$
commutes we say that $(A,\mu_A,\eta_A)$ is {\em co-associative}, and if the diagram
\begin{equation}\label{eq:unit}
\xymatrix{
A \ar@{=}[drr] && A\otimes A \ar[ll]_-{\eta_A \otimes Id} \ar[rr]^-{\Id \otimes \eta_A} && A \ar@{=}[dll] \\
&& A \ar[u]^{\mu_A} &&
}
\end{equation}
commutes we say that $(A,\mu_A,\eta_A)$ is {\em counital}.
Finally, $(A,\mu_A,\eta_A)$ is {\em cocommutative} if in addition the diagram
\begin{equation}\label{eq:co-commutative}
\xymatrix{
A \ar[d]_-{\mu_A} \ar[r]^-{\mu_A} & A\otimes A \\
A \otimes A \ar[ur]^-\simeq_{\tau_{A,A}} & \,,
}
\end{equation}
commutes, where $\tau_{A,A}$ stands for the symmetry constraint. A {\em coalgebra map} $f\colon (A,\mu_A,\eta_A) \to (B,\mu_B,\eta_B)$ consists of a map $f\colon A \to B$ in $\cC$ making the diagram
$$
\xymatrix{
 A \otimes A \ar[r]^-{f\otimes f} & B \otimes B \\
 A \ar[r]_-f \ar[u]^{\mu_A} & B \ar[u]_-{\mu_B}\,
}
$$
commutes.
When $\eta_B \circ f=\eta_A$ we say that it is {\em counit preserving}.

Now, let us assume that the symmetric monoidal category $(\cC,\otimes,{\bf 1})$ is moreover additive. Note that given counital coassociative coalgebras $(A,\mu_A,\eta_A)$ and $(B,\mu_B,\eta_B)$, the direct sum $A \oplus B \in \cC$ becomes a counital coassociative coalgebra. Its comultiplication is given by 
$$ \mu_{A\oplus B}\colon A \oplus B \stackrel{\mu_A \oplus \mu_B}{\too}(A \otimes A) \otimes (B\otimes B) \subset (A\oplus B) \otimes (A\oplus B)$$
and its counit is given by 
$$ \eta_{A\oplus B} \colon A \oplus B \stackrel{\eta_A\oplus \eta_B}{\too} {\bf 1} \oplus {\bf 1} \stackrel{\nabla}{\too}{\bf 1}\,.$$
Moreover, if $(A, \mu_A,\eta_A)$ and $(B,\mu_B,\eta_B)$ are cocommutative the same holds for the direct sum $A\oplus B$. We now have all the ingredients needed for the following result:
\begin{proposition}\label{prop:appendix}
Consider the following diagram in $\cC$
$$
\xymatrix{
C \ar@<-0.5ex>[r]_-f & A \ar@<-0.5ex>[l]_-r \ar@<-0.5ex>[r]_g  & B   \ar@<-0.5ex>[l]_-s\,.
}
$$
Assume that $r\circ f =\Id$, $g\circ s=\Id$ and $f\circ r + s\circ g =\Id$. Assume also that $A$ and $B$ are counital coassociative coalgebras and that $s$ and $g$ are coalgebra maps with $s$ counit preserving. Under these assumptions, $C$ becomes a counital coassociative coalgebra and $r$ and $f$ coalgebra maps with $f$ counit preserving. Moreover, if $A$ and $B$ are cocommutative the same holds for $C$. Furthermore, the induced isomorphism $[r\,\,\,g]\colon A \stackrel{\sim}{\to} C \oplus B$ is a counit preserving coalgebra map.
\end{proposition}
\begin{proof}
Let us start by constructing the comultiplication on $C$. Note that it follows from our hypothesis that $C$ identifies with the cokernel of $s$. Hence, consider the following commutative diagram
$$
\xymatrix{
B \otimes B \ar[r]^-{s \otimes s} & A\otimes A \ar[r] & \mathrm{cok}(s\otimes s) \ar[r] & \mathrm{cok}(s) \otimes \mathrm{cok}(s) \\
B \ar[u]^-{\mu_B} \ar[r]_s & A \ar[u]^-{\mu_A} \ar[r]_r & \mathrm{cok}(s) \ar@{-->}[u] \ar@{=}[r] & \mathrm{cok}(s) \ar[u]_-{=:\mu_C}\,.
}
$$
The left-hand-side square commutes since $s$ is a coalgebra map; the middle one commutes since the dashed arrow is induced by the universal property of the cokernel; and the right-hand-side one commutes since the composition 
$$B \oplus B \stackrel{s\otimes s}{\to} A \otimes A \stackrel{r\otimes r}{\to} \mathrm{cok}(s) \otimes \mathrm{cok}(s)$$
is trivial. The comultiplication $\mu_C$ on $C$ is then given by the vertical arrow on the right-hand-side. With this definition it is clear that $r$ becomes a coalgebra map. Note also that since $\mu_A$ and $\mu_B$ are co-associative the same holds for $\mu_C$. Similarly, if by hypothesis $\mu_A$ and $\mu_B$ are cocommutative the same holds for $\mu_C$.

Let us now prove that $f$ is also a coalgebra map. Consider the diagram:
\begin{equation}\label{eq:diag2}
\xymatrix{
C\otimes C \ar[r]^-{f \otimes f} & A \otimes A \ar[r]^-{g \otimes g} & B \otimes B \\
C \ar[u]^-{\mu_C} \ar[r]_-f & A \ar[u]^-{\mu_A} \ar[r]_-g & B \ar[u]_-{\mu_B}\,.
}
\end{equation}
One needs to show that the left-hand-side square is commutative. By hypothesis $g$ is a coalgebra map and so the right-hand-side square is commutative. Moreover, since $g \circ f =0$, the outer square is also commutative (since both maps from $C$ to $B \otimes B$ are trivial). This implies that the two maps
\begin{eqnarray*}
C \stackrel{\mu_C}{\to} C\otimes C \stackrel{f \otimes f}{\to} A \otimes A \stackrel{g \otimes g}{\to} B \otimes B && C \stackrel{f}{\to} A \stackrel{\mu_A}{\to} A \otimes A \stackrel{g \otimes g}{\to} B \otimes B 
\end{eqnarray*}
agree. Since $g \otimes g$ is surjective (note that it admits a section $s \otimes s$) we then conclude that  $(f \otimes f) \circ \mu_C =\mu_A \circ f$, \ie that the above left-hand-side square in \eqref{eq:diag2} commutes.

Let us now define the counit of $C$ as the composition $\mu_C \colon C \stackrel{f}{\to} A \stackrel{\eta_A}{\to} {\bf 1}$. Note that proving the commutativity of diagram \eqref{eq:unit} amounts to show that both composites 
\begin{eqnarray*}
C \stackrel{\mu_C}{\to} C\otimes C \stackrel{\Id \otimes f}{\to} C \otimes A \stackrel{\Id \otimes \eta_A}{\to} C &&
C \stackrel{\mu_C}{\to} C\otimes C \stackrel{f \otimes \Id}{\to} A \otimes C \stackrel{\eta_A \otimes \Id}{\to} C 
\end{eqnarray*}
are the identity. The proof is similar and so we restrict ourselves to the left-hand-side case. Consider the following commutative diagram
$$
\xymatrix{
C \ar[r]^-{\mu_C} & C \otimes C \ar[r]^{\Id \otimes f} & C \otimes A \ar[r]^-{\Id \otimes \eta_A} & C \ar@{=}[d] \\
A \ar[u]^-r \ar[r]^-{\mu_A} & A \otimes A \ar[u]_-{r \otimes r} \ar[rr]^-{r \otimes (\eta_A \circ f \circ r)} && C \\
C \ar[u]^-f \ar[r]_-{\mu_C} & C \otimes C \ar[u]_-{f\otimes f} & & \,.
}
$$
Since $r \circ f=\Id$ it suffices to show that the composite
\begin{equation}\label{eq:composite}
C \stackrel{f}{\to} A \stackrel{\mu_A}{\to} A \otimes A \stackrel{r\otimes(\eta_A \circ f \circ r)}{\too} C
\end{equation}
is the identity map. Since $f \circ r =(\Id-s\circ g)$, we have $r \otimes (\eta_A \circ f \circ r)=r\otimes \eta_A -r\otimes(\eta_A \circ s\circ g)$ and hence, since $(r \otimes \eta_A)\circ \mu_A=r$, we obtain the equality 
$$ (r \otimes (\eta_A \circ f \circ r)) \circ \mu_A = r - (r \otimes (\eta_A \circ s \circ g)) \circ \mu_A\,.$$
Now, note that $g\circ f=0$. Since $s$ is a monomorphism it suffices to show that $s\circ g\circ f=0$ which follows from the equalities:
$$s \circ g \circ f= s\circ g\circ (f\circ r \circ f) = (s \circ g) \circ (\Id - s\circ g)\circ f = (s \circ g - s\circ g) \circ f=0\,.$$
The equalities $g\circ f=0$ and $\mu_A\circ f =(f \otimes f) \circ \mu_C$ allows us then to conclude that 
$$ (r\otimes(\eta_A \circ f \circ r)) \circ \mu_A \circ f =r\circ f =\Id\,.$$
This shows that $C$ is also a counital coalgebra and that $f$ is counit preserving. 

Let us finally prove that the induced isomorphism $[r \,\,\,g]\colon A \stackrel{\sim}{\to} C \oplus B$ is a counit preserving coalgebra map. The commutative squares
$$
\xymatrix{
A \otimes A \ar[r]^-{r \otimes r} & C \otimes C  && A \otimes A \ar[r]^-{g \otimes g} & B \otimes B \\
A \ar[u]^-{\mu_A} \ar[r]_-r& C \ar[u]_-{\mu_C} && A \ar[u]^-{\mu_A} \ar[r]_g & B\ar[u]_-{\mu_B} 
}
$$ 
imply automatically that $[r\,\,\,g]$ is a coalgebra map. In what concerns the counit, one needs to prove that the following diagram commutes
$$
\xymatrix{
A \ar[dd]_-{\eta_A} \ar[r]^-{[r \,\,\,g]} \ar[r] & C \oplus B \ar[d]^-{\eta_C \oplus \eta_B}\\
&  {\bf 1} \oplus {\bf 1} \ar[d]^-{\nabla} \\
{\bf 1} \ar@{=}[r] & {\bf 1}\,.
}
$$
Note first that since by hypothesis $s$ is counit preserving we have $\eta_B \circ g =\eta_A \circ s \circ g$. On the other hand, by the above definition of $\eta_C$ we have $\eta_C \circ r=\eta_A \circ f \circ r$. As a consequence, we obtain the equality
$$ \eta_C \circ r + \eta_B \circ g =\eta_A(f\circ r + s \circ g) =\eta_A$$
and thus conclude that the above diagram commutes. This concludes the proof.
\end{proof}

\end{document}